\documentclass{psp}

\usepackage{amsmath}
\usepackage{amsfonts}
\usepackage{latexsym}
\usepackage{mathrsfs}
\usepackage[mathcal]{euscript}
\usepackage{amssymb}
\usepackage{pstricks}
\usepackage{color}                 
\definecolor{red}{cmyk}{0,1,1,0}

\ifprodtf \else
  \IfFileExists{amsbsy.sty}
    {\typeout{^^JFound the 'amsbsy' package on the system, using it.^^J}%
     \usepackage{amsbsy}}
    {}

\fi

\newtheorem{theorem}{Theorem}[section]
\newtheorem{lemma}{Lemma}[section]
\newtheorem{corollary}[theorem]{Corollary}

\newromanexpr\Hess{Hess}

\DeclareMathOperator{\di}{diam}
\DeclareMathOperator{\card}{\#}
\def\eg{{\it{e.g.}}}
\def\ie{{\it{i.e.}}}
\def\DA{Diophantine approximation}
\def\HD{Hausdorff dimension}
\def\hdim{\mathrm{dim_{\,H}}}
\def\HM{Hausdorff measure}

 \def\cB{{\mathcal B}}
\def\cC{{\mathcal C}}
 \def\cH{{\mathcal H}}
\def\sH{{\mathscr H}}
\def\fBQ{{\mathfrak{B}_{\mathbb{H}}}}
 \def\cR{{\mathcal R}}
 \def\cRq{{\mathcal R}_{\q}}
\def\cL{{\mathcal L}}
 \def\cW{{\mathcal W}}
 \def\cV{{\mathcal V}} 
\def\Ja{Jarn{\'\i}k}
\def\JB{Jarn\'{\i}k-Besi\-covitch}
\def\GI{Gaussian integer}
\def\K{Khintchine}

\def\BVt{Beresnevich-Velani theorem}
\def\KJ{Khintchine-Jarn{\'\i}k}

\def\Spr{Sprind{\v z}uk}
\def\sv{\mathbf{s}}
 \def\p{\mathbf{p}}
 \def\r{\mathbf{r}}
\def\pq{\p\q^{-1}}
\def\rs{\r\sv^{-1}}
 \def\q{\mathbf{q}}
\def\xb{\mathbf{x}}
\def\yb{\mathbf{y}}
\def\cb{\mathbf{c}}
 \def\quat{{\mathbb{H}}}
 \def\Qu{\mathbb{H}}
\def\ab{\mathbf{a}}
\def\bb{\mathbf{b}}
  \def\pb{\mathbf{p}}
 \def\qb{\mathbf{q}}
 \newcommand{\eH}{\eta}
 \newcommand{\rh}{\rho}
\def\BA{\mathfrak{B}}
\def\QBA{{\mathfrak B}_{\mathbb{H}}}
\def\eps{\varepsilon}
 \def\vpi{\varpi}
\DeclareMathAlphabet{\ams}{U}{msb}{m}{n}
\def\Z{\ams{Z}}\def\N{\ams{N}}
\def\H{\ams{H}}\def\R{\ams{R}}
\def\C{\ams{C}}\def\Q{\ams{Q}}
\def\ve{\varepsilon}
\def\hur{\mathcal H}
\def\hurat{\mathcal Q}
\def\psp{\mathbf{PSp}}
\def\r{\mathbf{r}}
\def\s{\mathbf{s}}
\DeclareMathAlphabet{\goth}{U}{euf}{m}{n}

\begin{document}

\title[Metrical Diophantine approximation for quaternions]
  {Metrical Diophantine approximation for quaternions}

\author[Maurice Dodson and Brent Everitt]{MAURICE DODSON\thanks{The
    first author is grateful to the 
    Royal Society for supporting a Visiting Fellowship to the
    University of Adelaide and its Mathematics Department for its
    hospitality.} \enskip\nobreakand\
  BRENT EVERITT\thanks{Some of the
    results of this paper were obtained while the second author was
    visiting the Institute for Geometry and its
    Applications, 
    University of Adelaide, Australia.  He is grateful for their
    hospitality.} \\
Department of Mathematics, University of York,\addressbreak
York, YO\textup{10 5}DD, UK}

\maketitle

\begin{center}
Dedicated to  J. W. S. Cassels.
\end{center}

\begin{abstract}
Analogues of the classical theorems of \K, \Ja{} and
    \JB{} in the metrical theory of \DA{} are established for
    quaternions by applying results on the measure of general `lim
    sup' sets. 
\end{abstract}

\section{Introduction}\label{sec:1}

\DA\ begins with a more quantitative understanding of the density of
the rationals $\Q$ in the reals $\R$. For any real number $\xi$, one
considers rational solutions $p/q$ to the inequality
$$
\left|\xi-\frac{p}{q}\right|<\ve,
$$
where $\ve$ is a small positive number depending on $p/q$.  Dirichlet's
theorem~\cite[Chap.~XI]{HW}, where $\ve=(qN)^{-1}$ for any $N\in \N$
and a suitable positive integer $q\leqslant N$, is fundamental to the
theory.  Holding for all real numbers, it is a {\em{global}} result in
\Spr's classification of \DA~\cite[pg.~x]{Sprindzuk}, in contrast with
{\em{individual}} results, which hold for special numbers, such as the
golden ratio $\phi$, $e$, $\pi$, {\em{etc}}., and with the
{\em{metrical}} theory. The last theory uses measure theoretic ideas to
describe sets of number theoretic interest and is the setting of this
paper.

Dirichlet's theorem underpins four major theorems -- or the \emph{Four
  Peaks} -- in the metrical theory of \DA{} for $\R$. These results
are concerned with the measure (usually Lebesgue or Hausdorff) of real
numbers that infinitely often are `close' to rationals, and those
which `avoid' rationals; these are called {\em{well}} approximable and
{\em{badly}} approximable numbers respectively (definitions
are given below). The four results are \K's theorem
(Theorem~\ref{thm:KT}), two theorems of \Ja{} (Theorems~\ref{thm:JT}
and~\ref{thm:JBAR}) the celebrated \JB{} theorem
(Theorem~\ref{thm:JBT}).  Three of the peaks concern well-approximable
numbers and one badly-approximable numbers. The quantitative form of
\K's theorem (see~\cite{Schmidt64,Sprindzuk}) certainly merits peak
status as well but will not be considered here.

This basic setting can be generalised in a number of directions: one
`topological', where the reals are replaced by $\R^n$ or even
submanifolds of $\R^n$ and the nature of the \DA{} modified
appropriately; another is `geometrical' where the reals
  are replaced by limit points of a discrete group acting on
  hyperbolic space; while yet another is `algebraic', where the field $\R$ is
replaced by other fields, skew-fields or division algebras, and $\Q$
is replaced by the field of fractions of `integral' subrings.  This
paper follows the third direction: the approximation of quaternions
$\H$ by ratios of integer-like quaternions. For us, `integer-like'
will mean the Hurwitz integers $\hur$: these turn out to be the
simplest subring of $\H$ with sufficiently nice algebraic properties
-- such as a division algorithm -- for an interesting quaternionic
number theory (see~\S\ref{subsec:preliminaries}
and~\cite{ConwaySmith,HurwitzZQ}).

As far as we can determine, little has been published on quaternionic
\DA. A.~Speiser obtained an approximation constant for
irrational quaternions~\cite{Speiser1932}; his work was extended and
sharpened by A.~L.~Schmidt~\cite{ASchmidt69,ASchmidt74}.  K.~Mahler proved an
inequality for the product of Hurwitzian integral linear
forms~\cite{Mahler45} but we can find nothing explicitly on the
metrical theory.  This paper sets out to fill this gap by establishing
quaternionic analogues of the Four Peaks.  Limitations of space and
complications arising from non-associativity prevent including the
further extension to octonions and completing the picture for real
division algebras.

After some basic measure theory in~\S\ref{sec:MaD} and a brief survey
of real, complex and more general \DA\ in \S\ref{sec:4peaksinRC},
we set the stage for the quaternionic theory in~\S\ref{sec:QDA}. The
main result here is a quaternionic analogue of Dirichlet's theorem
(Theorem~\ref{thm:QDT}). The badly approximable quaternions are then
defined in~\S\ref{subsec:QBA}.  Section~\ref{sec:Psiapproxquats}
extends the fundamental Dirichlet inequality to the notion of
$\Psi$-approximability.  The ideas of resonance and near-resonance are
explained and the basic structure of the set of $\Psi$-approximable
numbers is described.

We are finally ready for the quaternionic Four Peaks
in~\S\ref{sec:4peaksinQ}.  The First Peak is the quaternionic
Khintchine theorem (Theorem~\ref{thm:qKT}).  Each of the statement and
proof falls into two cases: convergence and divergence. The
convergence case is the easier of the two and is established
in~\S\ref{sec:Kconvergent}.  The divergence case, proved
in~\S\ref{sec:divergentsum}, is much harder and requires deeper ideas,
such as ubiquity~(\S\ref{sec:ubiquity}) and the mass transference
principle (\S\ref{subsec:mtp2}).  The quaternionic 
Dirichlet's theorem (Theorem~\ref{thm:QDT}) is used
in~\S\ref{sec:divergentsum} to show that the Hurwitz rationals $\hurat$
are a ubiquitous system. This involves rather lengthy and delicate
analysis but is a prerequisite to applying the powerful
Beresnevich-Velani Theorem, established in~\cite{BVParis2009a}. This
is adapted to our needs as Theorem~\ref{thm:qbv} and used to
yield the analogue of \K's theorem. The extension to quaternions of
the quantitative form of \K's theorem is an interesting open question.

  The proof of Theorem~\ref{thm:qJT}, the quaternionic analogue of the
  \Ja's extension of \K's theorem to Hausdorff
  measure, follows similar lines and is sketched.
  Theorem~\ref{thm:qJBT}, the analogue of the Jarn\'ik-Besicovitch
  theorem, is a corollary of Theorem~\ref{thm:qJT}. Finally, some
  related ideas are used in~\S\ref{sec:BAquats} to prove
  Theorem~\ref{thm:QJTBA} on the Hausdorff measure and dimension of
  the set of badly approximable quaternions.

  A knowledge of measure theory and particularly Lebesgue and
  Hausdorff measure in $\R^k$ will be assumed.  For completeness and
  to fix notation the elements of the theory are sketched. The reader
  is referred
  to~\cite{MDAMshort,FalcGFSshort,FalcFG,Fed,MattilaGS,Rogers} for
  further details.

\section{Measure and dimension} 
\label{sec:MaD}
We consider points in subsets of general Euclidean space $\R^n$, our
primary interest of course being in $\R^4$, the underlying set of
$\H$.  When defined, the Lebesgue measure of a set $E$ will be denoted
by $|E|$. The set $E\subseteq F \subseteq \R^n$ is said to be
{\it{null}} if $|E|=0$ and {\it{full in $F$}} if its complement
$ F\,\setminus \, E$ is null (reference to $F$ will be
omitted when there is no risk of ambiguity).  Hausdorff measure and
Hausdorff dimension are much more general and can be assigned to any
set.  In particular they can be applied to different null sets (also
referred to as {\em{exceptional sets}}), so offering a possible means
of distinguishing between them.

A {\it{dimension function}} $f\colon [0,\infty)\to [0,\infty)$ is a
generalisation of the usual notion of dimension; $m$-dimensional
Lebesgue measure corresponds to $f(t)=t^m$.  More generally, the
function $f$ will be taken to be increasing on $[0,\infty)$, with
$f(x)>0$ for $x>0$ and $f(x)\to 0$ as $x\to 0$. For convenience $f$
will be assumed to be continuous, so that $f(0)=0$.  The
{\em{Hausdorff $f$-measure}} $\sH^f$ (or generalised Hausdorff measure
with dimension function $f$) is defined in terms of a $\eps$-cover
$\cC_\eps=\{C_i\}$ of a set $E$, so that $ E\subseteq \bigcup_{i=1}^\infty
C_i, $ where $ \di(C_i) \leqslant \eps $.  The measure $\sH^f(E)$, 
defined as
\begin{equation}
\label{eq:Hfmeasure}
\sH^f(E): = \lim_{\eps\to 0} \inf \{\sum_i f(\di(C_i))  
\colon   C_i\in \cC_\eps \}, 
\end{equation} 
is a Borel measure and regular on Borel
sets~\cite{FalcGFSshort,MattilaGS}. Hausdorff $s$-measure $\sH^s$ corresponds
to the function $f$ being given by $f(t)=t^s$, where $0\leqslant s <\infty$.
When $s=m$ a non-negative integer, Hausdorff $s$-measure is comparable
with Lebesgue's $m$-dimensional measure. Indeed
\begin{equation}
  \label{eq:comparable}
  \sH^m(E) = 2^m |B(0,1)|^{-1} |E|,
\end{equation}
where $B(0,1)$ is the unit $m$-dimensional ball, and the two measures
agree when $m=1$~\cite[pg.~56]{MattilaGS}.  The 4-dimensional Lebesgue measure
(4-volume) of the $4$-ball $B^{(4)}(\xi,r)$ of radius $r$ (and
diameter $2r$) centred at $\xi$ is given by
\begin{equation}
  \label{eq:4ballmeasure}
|B(\xi,r)|=\frac{\pi^2}{2} r^4 \asymp r^4.  
\end{equation}

 For each set $E$ the {\em{Hausdorff
    dimension}} $\hdim E$ of $E$ is defined by
$$
\hdim E := \inf \{s \in \R \colon \sH^s(E) = 0 \},
$$
so that
\begin{equation*}
\sH^s(E) =   \begin{cases}  \infty, & s < \hdim\, E, \\
                   0,     & s > \hdim\, E. 
  \end{cases}
  \end{equation*}
  Thus the dimension is that critical value of $s$ at which $\sH^s(E)$
  `drops' discontinuously from infinity. Hausdorff dimension has the
  natural properties of dimension. For example, if $E\subseteq E'$,
  then $\hdim \, E\leqslant \hdim \, E'$; and an open set, or a set of
  positive Lebesgue measure in $\R^n$, has maximal or full Hausdorff
  dimension $n$.  Different null sets can have different Hausdorff
  dimension and so can be distinguished (\eg, Theorem~\ref{thm:JBT}).

  The Hausdorff $s$-measure at the critical point can be $0, \, \infty$
  or any intermediate value.  Methods for determining the Hausdorff
  dimension, such as the regular systems given in~\cite{BS} or the 
  ubiquitous systems of~\cite{DRV90a}, do not specify the $s$-measure
  at the critical point in general and a deeper approach is usually
  needed (see Theorem~\ref{thm:BV}).  In the case of lim sup sets,
  such as the $\Psi$-approximable numbers defined below, the measure
  of a natural cover arising from the definition leads to a sum which
  determines the Hausdorff measure at the critical point.

\section{Real and complex metrical \DA }
\label{sec:4peaksinRC}

Some of the salient features of metrical \DA{} for the real and complex
numbers are set out to aid comparison with the quaternions. 

\subsection{Metrical \DA{} for real numbers}
\label{subsec:mDAforR}

Historically,
metrical \DA{} began with Borel's study of the set
\begin{equation*}
  W_v:=\left\{\xi\in\R\colon \left|\xi - \frac{p}{q}\right|<q^{-v}
    {\text{ for infinitely many }} p \in \Z,\, q\in\N \right\},  
\end{equation*}
where $W_v=\R$ for $v\leqslant 2$ and is null for $v>2$~\cite{Borel12}.
More generally, the function $x\mapsto x^{-v}$ is replaced by an
{\it{approximation function}} $\Psi$, defined here to be a function
$\Psi\colon (0,\infty)\to (0,\infty)$ with $\Psi(x)\to 0$ as
$x\to\infty$.  One studies the Lebesgue measure $|W(\Psi)|$ of the set
\begin{equation*}
\label{eq:WPsi}
W(\Psi):=\left\{\xi\in\R\colon \left|\xi - \frac{p}{q}\right|<\Psi(q)
{\text{ for infinitely many }}  p \in \Z, \, q\in\N \right\}
\end{equation*}
of {\em{$\Psi$-approximable numbers}}. Unless otherwise stated, the
approximation function $\Psi$ will be taken to be decreasing
 (which we will take to mean non-increasing).
 
 For technical reasons, it is often better to work within a compact
 set and we choose the subset $V(\Psi):=W(\Psi)\cap [0,1]$. There is no
 loss in generality since $\R$ is the union of integer translates of
 $[0,1]$, the integers $\Z$ are a null set and Lebesgue measure is
 translation invariant, allowing the measure of $W(\Psi)$ to be
 deduced from that of $V(\Psi)$.  In particular, $V(\Psi)$ is full
 (in $[0,1]$) iff $A(\Psi)$ is full (in $\R$).  
Four of the principal results -- the
 Four Peaks -- in the theory for $\R$ now follow.

\subsection{The Four Peaks in the theory of real metrical \DA.} 
\label{subsection:classicalmDA}

\paragraph{The First Peak: Khintchine's theorem for $\R$.}
In 1924 Khintchine introduced a `length' criterion that gave a
strikingly simple and almost complete answer to the `size' of
$W(\Psi)$~\cite{Kh24}, extended to $\R^n$ (simultaneous Diophantine
approximation) in~\cite{Kh26}. The conditions on $\Psi$ have been
improved since (see for example~\cite[Ch.~VII]{Casselshort},
\cite[Ch.~1]{Sprindzuk} and~\S\ref{subsec:sdaR4}) to give the
following result for $\R$:
\begin{theorem}  
  \label{thm:KT}
    Let $\Psi\colon (0,\infty)\to (0,\infty)$. Then
  \begin{equation*}
    W(\Psi) \ and \ V(\Psi)\ are \ \begin{cases}
    \text{null} &\hspace{-.05in}\text{when } \sum_{m=1}^\infty m\,\Psi(m) <\infty, \\
      \text{full}   &\hspace{-.05in}\text{when $\Psi$ is decreasing and }  
      \sum_{m=1}^\infty m\,\Psi(m) =\infty \\
    \end{cases}   
  \end{equation*}
\end{theorem}

Note that $W(\Psi)$ being full implies the weaker statement that
$|W(\Psi)|=\infty$, while $|V(\Psi)|=1$ is equivalent to $W(\Psi)$
being full.  Other approximation functions can be used: \eg,
$\psi(x)=x\Psi(x)$, where $\|\xi\|$ is the distance of $\xi$ from the
nearest integer, which allowing the inequality to be expressed in the
concise form $\|q\xi\|<\psi(q)$ (\eg,~\cite{MDAMshort,Casselshort}),
with the numerator $p$ suppressed, while Dennis Sullivan
in~\cite{Sullivan82} uses $a(x)=x^2\Psi(x)$ (he also uses an
equivalent integral criterion instead of the sum $\sum_{m\in \N}
a(m)/m$).  The subset $W'(\Psi)\subset W(\Psi)$ of points $\xi$
approximated by rationals $p/q$ with $p,q$ coprime will not be
considered.

It is evident that the value of the sum 
\begin{equation*}
  \label{eq:critsumL1}
  \sum_{m=1}^\infty m\,\Psi(m)
\end{equation*}
in Theorem~\ref{thm:KT} determines the Lebesgue measure of $W(\Psi)$
and so will be called a {\em{critical}}  sum for
$W(\Psi)$.  Note that if the above critical sum
converges, $\Psi$ must converge to 0 and moreover there is no need in
this case for $\Psi$ to be monotonic.  \K's theorem is related to the
`pair-wise' form of the Borel-Cantelli Lemma
(see~\cite{BDV06,md09a,HarmanMNT}) which also falls into two cases
according as a certain sum of probabilities converges or diverges.

  The interpretation of the rationals as the orbit of a point at
  infinity under the action of the modular group provides a powerful
  geometrical approach to \DA{} in the reals
  (\eg,~\cite{Rankin57,Series85}) and more
  generally~\cite{AhlforsMT,BeardonGDG,NichollsETDG,Patterson76a}.
  It was the basis of Sullivan's proof~\cite[Th.~3]{Sullivan82} of a
  slightly stronger form of \K's theorem and more
  (see~\S\ref{subsection:meDAinC} below).  

\paragraph{The Second Peak: \Ja's Hausdorff $f$-measure theorem for $\R$.}
In 1931, \Ja\ obtained \HM{} results for simultaneous \DA{} in $\R^n$,
providing a more general measure theoretic picture of the sets
involved~\cite{Ja31} (see also~\cite[pg.~3]{BDV06}).  This
did not include Lebesgue measure which is excluded by a growth
condition on $f$ at 0.  Although originally proved for $\R^n$, \Ja's
result is again stated for the case $n=1$, with some unnecessary
monotonicity conditions omitted.

\begin{theorem}
  \label{thm:JT} 
Let $f$ be a dimension function such
  that 
$f(x)/x\to \infty$ as $x\to 0$ and 
$f(x)/x$ decreases as $x$ increases.  Then
  \begin{equation*}
    \sH^f(W(\Psi))=\sH^f(V(\Psi))=
    \begin{cases}
      0 &\hspace{-.05in}\text{when } \sum_{m=1}^\infty m \,f(\Psi(m)) <\infty, \\
      \infty  &\hspace{-.05in}\text{when $\Psi$ is decreasing and }
\sum_{m=1}^\infty m \, f(\Psi(m)) =\infty.
    \end{cases}
  \end{equation*} 
\end{theorem}

The condition $f(x)/x\to \infty$ as $x\to 0$ means that \Ja's theorem
does not imply Khintchine's theorem since the dimension function $f$
for 1-dimensional Lebesgue measure is given by $f(x)=x$.  However,
using the idea of ubiquity (explained below in~\S\ref{sec:ubiquity}),
V.~V.~Beresnevich and S.~L.~Velani have united Theorem~\ref{thm:KT}
and \Ja's theorem into a single general `\K-\Ja'
theorem~\cite[\S2.3]{BVParis2009a}.  The sum $\sum_{m=1}^\infty m
\,f(\Psi(m))$ is the corresponding critical sum.

\paragraph{The Third Peak: the \JB{} theorem for $\R$.}
In 1929 \Ja~\cite{Ja28,Ja29} obtained the \HD{} of the set $W_v$,
proved by Besicovitch independently in 1934~\cite{Bes34}.  This result
is readily seen as a consequence of \Ja's result above by putting
$f(x)=x^{s}$ and $\Psi(x)=x^{-v}$, $v>0$.

 \begin{theorem}
\label{thm:JBT}
 Let $v\geqslant 0$. Then the \HD{} of $W_v$ is given by
 \begin{equation*}
    \hdim W_v =  \hdim V_v = 
    \begin{cases}
      1  & \text{ when } v\leqslant 2, \\
      \dfrac{2}{v} & \text{ when } v\geqslant 2.
    \end{cases}
  \end{equation*}
\end{theorem}
When $1/\Psi$ has lower order $\lambda(1/\Psi):= \liminf_{N\to\infty}
(\log 1/\Psi(N))/(\log N )$, then
\begin{equation*}
\label{eq:dimlo}
  \hdim W(\Psi) =  \hdim V(\Psi) = 
\begin{cases}
     \ \ \ \ 
 1  & \text{ when } \lambda(1/\Psi)\leqslant 2, \\
      \dfrac{2}{\lambda(1/\Psi)} & \text{ when } \lambda(1/\Psi) \geqslant 2
    \end{cases}
\end{equation*}
(see~\cite{mmd92, DRV90a}).

\paragraph{The Fourth Peak: \Ja's theorem for ${\mathfrak B}$, the set
  of badly approximable numbers.} A real number $\beta$ is said to be
{\em{badly approximable}} if there exists a constant $c=c(\beta)$ such
that
      \begin{equation*}
              \left|\beta - \frac{p}{q}\right|\geqslant \frac{c}{q^2}
      \end{equation*}
      for all rationals $p/q$. The set of badly approximable numbers
      is denoted by $\BA$ and can be regarded as a `lim inf'
      set~\cite[pg.~1]{FalcGFSshort}. In his pioneering paper of 1928, 
      \Ja{} established the Lebesgue measure and Hausdorff dimension
      of $\BA$~\cite{Ja28}.
  \begin{theorem} 
  \label{thm:JBAR} 
    The set $\BA$ is null with full \HD, \ie, $|\BA|=0 $ and $\hdim
    ({\mathfrak B}) = 1$.
   \end{theorem}
The strengthening of this result by W.~M.~Schmidt, who showed that
   $\BA$ was a `winning set' in a certain game~\cite{Schmidt66}, will
   not be considered for quaternions.

\subsection{Metrical \DA{} for the complex numbers}
\label{subsection:meDAinC}

Approximating complex numbers by ratios of \GI s $\Z[i]$, a half way
house to approximating quaternions by ratios of Lipschitz or Hurwitz
integer quaternions, was studied by Hermite and Hurwitz in the 19th
century~\cite[Chapter~IV,\S~1]{Koksma}. Continued fractions for
complex numbers, so simple and effective for real numbers, turn out to
be much more difficult than the real
case~\cite{Casselshort,HW,SchmidtDA,Ford1925}.  In the 1950s, Farey
sections for complex numbers, analogous to Farey fractions for real
numbers, were developed by Cassels, Ledermann and Mahler, who carried
out a detailed study~\cite{clm} of a programme sketched out by
Hurwitz~\cite[\S8]{Hurwitz1891}; their work was simplified and
extended by LeVeque~\cite{leveque1952}.  A.~L.~Schmidt developed a
natural and effective approach in~\cite{ASchmidt67,ASchmidt75a},
subsequently extended to the even more difficult case of
quaternions~\cite{ASchmidt69,ASchmidt74}.

Each of the Four Peaks has an analogue in the complex numbers. That of
Khintchine's theorem was by proved by LeVeque~\cite{leveque1952}, who
combined Khintchine's continued fraction approach with ideas from
hyperbolic geometry.  Later Patterson, Sullivan and others made full
use of groups acting on hyperbolic space to prove \DA{} results in
more general settings.  Sullivan established a \K{} theorem for \DA{}
in the imaginary quadratic fields $\Q(\sqrt{-d})$, where $d$ is a
positive non-square integer~\cite[Theorem~1]{Sullivan82},
corresponding to the Bianchi groups.  In the case $d=1$, the field is
the complex numbers, corresponding to the Picard group, and Theorem~1
in~\cite{Sullivan82} reduces to the complex analogue of \K's theorem.

The Mass Transference Principle (see~\S\ref{subsec:mtp2})
  could be applied to the complex analogue of \K's theorem to deduce
  the complex analogue of Theorem~\ref{thm:JT} (indeed more general
  analogues involving Bianchi groups could be deduced from the more
  general analogues of \K's theorem).
The complex \JB{} theorem and a stronger form of \Ja's
Theorem for badly approximable complex numbers were also proved
in~\cite{dk04a}, using respectively ubiquity (Theorem~6.1; see
also~\cite[Cor.~7]{BDV06}) and the $(\alpha,\beta)$
games of  W.~M.~Schmidt~\cite{Schmidt66} (Theorem~5.2).

\subsection{Generalisations}
\label{subsec:generalisations}

The above results, with appropriate modifications, hold for
simultaneous \DA{} and more generally for
systems of linear forms (where the \K-Groshev theorem takes the place
of \K's theorem)~\cite{HDSV97,Schmidt69,Sprindzuk}.  The \K-Groshev
theorem was extended to non-degenerate manifolds in the case of
convergence by Beresnevich, D.~Y.~Kleinbock and G.~A.~Margulis
in~\cite{bkm01,vvb02} and in the case of divergence by the preceding
authors and V.~I.~Bernik in~\cite{bbkm02}.  The idea of
ubiquity~\cite{DRV90a}, which is closely related to regular systems,
has been extended by Beresnevich, H.~Dickinson and Velani to lim sup
sets in compact metric spaces supporting a suitable non-atomic measure
to create a broad unifying theory~\cite{BDV06,BV06, BVParis2009a}.  In
particular, the results in~\cite{BDV06} imply that the measure in the
\BVt~\cite[Th.~3]{BVParis2009a} covers both Lebesgue and Hausdorff
measure and will be applied to establish the first three of the
quaternionic Four Peaks.

The approach using discrete group actions on hyperbolic space for $\R$
and $\C$, already alluded to above, leads naturally to the more
general setting of Kleinian group actions on hyperbolic space
(\cite{BJ97} has a comprehensive list of references).  Beresnevich,
Dickinson \& Velani have established metrical \DA{} results for more
general Kleinian group analogues of the first three of the Four
Peaks~\cite{BDV06}. These specialise to metrical \DA{} results for
real and complex numbers, corresponding to the modular and Picard
group respectively.  The quaternionic case would correspond to the
group $\psp_{2,1}(\hur)$ but different normalisations require
reconciling and the proofs would also require a knowledge of the
theory of discrete groups acting on (quaternionic) hyperbolic space.
A more direct and less abstract approach is taken in this paper.

  In a continuation of~\cite{BV06}, S.~Kristensen, R.~Thorn \&
  Velani~\cite{ktv06} extend the definition of a badly approximable
  point to a metric space.  This allows the metrical
  structure of $\QBA$, the quaternionic analogue of badly approximable
  points, to be read off once a few geometric conditions are verified
  (see~\S\ref{subsec:QBA}). It could  also be possible to
    use the equivalence of badly approximable points and `bounded'
    orbits (for details see~\cite{DaniNTDS}).  `Divergent' orbits
    correspond to well-approximable points but the results are less
    precise~\cite{Dani85}.

\section{Quaternionic \DA} 
\label{sec:QDA} 
We begin our study of quaternionic \DA{} by identifying the
appropriate analogues of the classical case and then proving an
analogue of Dirichlet's theorem.

\subsection{Preliminaries on quaternionic arithmetic}
\label{subsec:preliminaries}

The skew field $\Qu$ of quaternions consists of the set 
\begin{equation*}
 \{\xi=a+ bi + cj +dk\colon a,b,c,d\in \R\},
\end{equation*}
subject to $i^2=j^2=k^2=ijk=-1$ and $i,j,k$ anticommuting: $ij=-ji,
jk=-kj$ and $ik=-ki$.  The norm of a quaternion $\xi$ is
  taken to be the usual Euclidean norm 
\begin{equation*}
   |\xi|_2:= (\xi\overline{\xi})^{1/2}=
(|a|^2+ \dots  +|d|^2)^{1/2},
\end{equation*}
where $\overline{\xi}= a-bi-cj-dk$.  This norm is multiplicative,
with $|\xi\xi'|_2 = |\xi|_2 |\xi'|_2 = |\xi'|_2 |\xi|_2=|\xi'\xi|_2 $
for $\xi,\xi'\in\Qu$ (in~\cite[\S\S20.6--20.8]{HW}
and~\cite{HurwitzZQ}  `norm' is used in a different sense, with
$N(\xi)=|\xi|_2^2$). 
  When convenient, we will write $\xi=a+ bi + cj +dk
=(a,b,c,d)$ and $a=\Re(\xi)$.

There are 24 multiplicative units in $\H$:
\begin{equation*}
\pm 1,\pm i,\pm j, \pm k\text{ and }
\pm\frac{1}{2}+\pm\frac{1}{2}i+\pm\frac{1}{2}j+\pm\frac{1}{2}k,
\end{equation*}
forming the vertices of a regular $24$-cell in $\R^4$.

The simplest-minded notion of \emph{integers\/} in $\H$ is that of the
Lipschitz integers $\cL=\Z[i,j,k]=\Z+i\Z+j\Z+k\Z \cong \Z^4$.
However, this choice has a number of shortcomings: it does not include
all the $\H$-units and is not a Euclidean domain, as the centre of a
$4$-dimensional cube has Euclidean distance $1$ from the closest
integral points (to be an integral domain, the distance of a
quaternion to the closest integral point should always be $<1$).  For
these reasons, the usual choice for quaternionic integers is the set
$\hur$ consisting of the quaternions $a+bi+cj+dk$, where either all of
$a,b,c,d\in\Z$ or all $a,b,c,d\in\Z+\frac{1}{2}$, \ie,
$$
  \hur=\cL\ \cup\,
\left(\cL + \frac12(1+i+j+k)\right).
$$
Thus $\hur$ consists of $\Z^4$ together with the mid-points of the
standard 4-dimensional unit cubes in $\Z^4$ and is an integral domain
with division algorithm, \ie, if $\p,\q\in \hur$ with $\q\ne 0$, then
there exist $\s,\r\in\hur$ with $|\r|_2<|\q|_2$ and
\begin{equation*}
  \p=\s\q + \r,
\end{equation*}
(for example, see~\cite[Th.~373]{HW}).  As a result, up to a
multiplicative unit, any two Hurwitz integers have a unique greatest
right (respectively left) common divisor up to a left (resp. right)
unit, whence Hurwitz integers have essentially a unique
factorisation~\cite{ConwaySmith}.  Two Hurwitz integers $\q,\q'$ are
said to be right (or left) coprime when their right (or left) greatest
common divisor is a unit; we will write $(\q,\q')_r=1$. Two coprime
integers $\q,\q'$ generate $\hur$ in the sense that $\hur$ is a sum of
the principle ideals they generate, \ie, $\q\hur+\q'\hur=\hur$. A
{\em{prime}} quaternion $\xi$ is divisible only by a unit and an
associate of $\xi$, \ie, if in the factorisation $\xi = \q\,\q'$,
either $\q$ or $\q'$ is a unit.  Prime integer quaternions have a
neat characterisation (modulo units) in terms of rational primes:
an integer quaternion $\xi$ is prime if and only if $|\xi|_2^2$ is a
rational prime \cite[Th.~377]{HW}.

As a subgroup of $\R^4$, the Hurwitz integers $\hur$ are free abelian
with generators $\{ i, j, k,$ $\frac12(1+i+j+k)\}$ and form a scaled copy
of the lattice spanned by the root system of the simple Lie algebra
$\goth{f}_4$.  A fundamental region for $\hur$ is given by the
half-closed region in $\R^4$ with vertices $0,1,i,j$ and
$\frac{1}{2}(1+i+j+k)$.  It has 4-dimensional Lebesgue measure, or
4-volume, $|\Delta|= 1/2$.  For convenience we choose the simpler region
\begin{equation}
  \label{eq:Delta}
  \Delta=\{\xi\in \Qu\colon 0 \leqslant a,b,c<1, 0 \leqslant d<1/2\}
=[0,1)^3\times [0,1/2).  
\end{equation}

The Hurwitz rationals $\hurat$ are defined to be
\begin{equation*}
  \label{eq:Hrats}
  \hurat:=\{\pq\colon \p, \q\in \hur, \q\ne 0\}.
\end{equation*}
A quaternion is said to be {\em{irrational}} if at least one of its
(real) coordinates is irrational.  Approximating quaternions by
Hurwitz rationals $\pq\in\hurat$, where the Hurwitz integer $\q$ can
be regarded as a `denominator' of the Hurwitz rational $\pq$, is an
obvious analogue of approximating a real number by rationals
$p/q\in\Q$.  Distinct Hurwitz rationals enjoy essentially the same
`separation' property as distinct rationals.
\begin{lemma}
\label{lem:distdiff}
  If $\pq\ne \rs$, then
\begin{equation*}
  |\pq-\rs|_2\geqslant |\q|_2^{-1}|\sv|_2^{-1}. 
\end{equation*}
 \end{lemma}
\noindent 
On expanding $|\pq-\rs|_2^2\,
\q\,\overline{\q}\,\s\,\overline{\s}$ and multiplying out, one
gets
\begin{equation*}
0<  |\pq-\rs|_2^2\,|\q|_2^2\,|\s|_2^2 = |\p|_2^2\,|\s|_2^2 +
|\r|_2^2\,|\q|_2^2 - 2\Re (\p\,\overline{\q}\,\s\,\overline{\r})\in\N
\end{equation*}
and the lemma follows.

\paragraph{}
\emph{For the rest of this paper, 
$\p$ and $\q$ will denote Hurwitz  integers with
$\q\ne 0$ unless otherwise stated}.

\subsection{Dirichlet's theorem for quaternions} 
\label{subsec:QDT}

The quaternions $\Qu = \bigcup_{\q\in\hur} (\Delta + \q)$, the union
of translates of the fundamental region $\Delta$.  Hence for any
$\xi\in \Qu$ and any non-zero $\q\in\hur$, there exists a unique
$\p=\p(\xi,\q)\in \hur$ such that $\{\xi\}_\Delta$, the {\em{Hurwitz
    fractional part}} of $\xi$ (the analogue of the fractional part
$\{\alpha\}$ of a real number $\alpha$), satisfies
\begin{equation*}
    \{\xi\}_\Delta := \xi- \p\in \Delta, 
  \end{equation*} 
so that 
\begin{equation*}
|\{\xi\}_\Delta |=    |\xi\q-\p|_2 \leqslant \frac{\sqrt{13}}{4} < 1. 
\end{equation*}

This inequality can be strengthened by restricting the choice of $\q$
to give a quaternionic version of a uniform Dirichlet's theorem, where
the approximation is by Hurwitz rationals $\hurat$ with the Euclidean
norm. A short geometry of numbers proof is given; it will be used in
Lemma~\ref{lem:fQubi}.  The multiplicative constant $2$
in~\eqref{eq:QDT2} is chosen for convenience and could be replaced any
number greater than $4/\pi$ $2$ without affecting the results sought. 
Whether $4/\pi$ is best possible is an open question. 
\begin{theorem}
  \label{thm:QDT}
  Given any $\xi\in \mathbb{H}$ and any integer $N>1$,
  there exist $\p, \q\in \hur$ with $1\leqslant |\q|_2 \leqslant N$ 
  such that
  \begin{equation}
    \label{eq:QDT2}
    \left\vert \xi-\p\,\q^{-1}\right\vert_2 < \dfrac{2}{\vert \q\vert_2 N}.
  \end{equation}
  Moreover there are  infinitely many $\p, \q\in \hur $ 
such that
  \begin{equation}
    \label{eq:QDio2}
    \left\vert \xi- \p\,\q^{-1}\right\vert_2  
    <  \dfrac{2}{\vert \q\vert_2^2}. 
  \end{equation} 
\end{theorem}
\begin{proof}  
We seek non-zero $\p,\q\in\hur$ as components for vectors in the set 
 \begin{equation*}
  K=\left\{
\left(\begin{array}{c}
  \xb \\ \yb
\end{array}
\right) 
\in\Qu^{\,2}\colon
 |\xi\yb-\xb|_2<\eps, |\yb|_2\leqslant N\right\} .
\end{equation*}
Now the set $K$ is convex and 
\begin{equation*}
   T(K)=  \left\{
\left(\begin{array}{c}
  \xb \\ \yb
\end{array}
\right) 
\in\Qu^2\colon
 |\xb|_2<\eps, |\yb|_2\leqslant N\right\} =B(0,\eps)\times B(0,N),
  \end{equation*}
where the matrix  $T=
\left(\begin{array}{cc}
  -1 & \xi \\ 0 & 1 
\end{array}\right)  $ has determinant $\det T=-1$ and $|T(K)|=|\det T|\, 
|K|=|K|$, 
the 8-volume of $K$.  Hence 
\begin{equation*}
  |K|=|B(0,\eps)|\times |B(0,N)| =
  \frac{\pi^2}{2}\eps^4\,\frac{\pi^2}{2}N^4 
= \frac{\pi^{4}}{4}\eps^4N^4.  
\end{equation*}

 The 4-volume of a fundamental region $\Delta$ of the Hurwitz
lattice is 1/2, so the 8-volume of $\Delta^2$ in $\Qu^2$ is 1/4. Hence
by Minkowski's theorem~\cite[Theorem 447]{HW}, if $|K| =
\pi^4\eps^4N^4/4 >2^8/4$, \ie, if $\eps>4/(\pi\,N)$, then $K$
contains a non-zero lattice point $(\p,\q)$ with $|\q|_2\leqslant N$ and
$|\xi\q-\p|_2<\eps$.  Choosing $\eps=2/N>4/(\pi\,N)$ gives
 \begin{equation*}
     \left\vert \xi-\p\,\q^{-1}\right\vert_2 < 
\frac{\eps}{\vert \q\vert_2} =\frac2{\vert \q\vert_2N}, 
   \end{equation*}
   where $|\q|_2\leqslant N$, which is~\eqref{eq:QDT2}.

   To show that there are infinitely many pairs $\p$, $\q$ in $\hur$
   satisfying~\eqref{eq:QDT2}, observe that the quaternionic rationals
   are not required to be in lowest terms, so that when $\xi=\ab
   \bb^{-1}$, $|\xi-\ab\bb^{-1}|_2= |\ab\bb^{-1}
   -\ab\,\p(\bb\,\p)^{-1}|_2=0$ for all non-zero $\p\in \hur$. Thus
   the inequality~\eqref{eq:QDio2} holds for infinitely many pairs
   $\ab\,\p$, $\bb\,\p$.  Note that if $\p,\q$ are coprime and $\xi =\ab
   \bb^{-1}\ne \pq$, where $\ab,\bb \in \hur$, then
   $|\pq-\ab\bb^{-1}|_2 \geqslant (|\q||\bb|)^{-1}$ by
   Lemma~\ref{lem:distdiff}, so that $|\bb|<|\q|$ and there are only
   finitely many solutions for~\eqref{eq:QDio2}.
 
The case when $\xi$ is not a Hurwitz rational remains, \ie, $\xi\ne
\ab \bb^{-1}$ for any $\ab,\bb\in\hur$, so that for all $\pq$, 
$|\xi-\p\,\q^{-1}|_2>0$.  
Suppose that the inequality~\eqref{eq:QDio2}
holds only for $\p\q^{-1}=\p^{(m)} (\q^{(m)})^{-1}$, where
$m=1,\dots,n$ and $|\q^{(m)}|_2\leqslant N$.  Then
\begin{equation*}
0<\min \left\{\frac{|\q^{(m)}|_2}{2} 
\left|\xi-\p^{(m)}(\q^{(m)})^{-1}\right|_2 \colon j=1,\dots,n\right\}=\eta 
\end{equation*}
for some $\eta>0$.  Let $N=[ 1/\eta ]+1>1/\eta$, where $[x]$ is the
integer part of the real number $x$.  Then by~\eqref{eq:QDT2}, there exist
$\p',\q'\in \hur$ with $|\q'|_2\leqslant N$ such that
  \begin{equation*}
    \left\vert \xi-\p'(\q'^{-1})\right\vert_2  < \dfrac{2}{\vert
    \q'\vert_2 \,N},
  \end{equation*}
whence
\begin{equation*}
   \frac{|\q'|_2}{2}\left\vert\xi-\p'(\q'^{-1})\right\vert_2
<\frac1{N} <\eta,
\end{equation*}
and $\p'(\q')^{-1}$ cannot be one of the
$\p^{(m)}(\q^{(m)})^{-1}$. This contradiction implies the
result.
\end{proof}

 The smaller constant
\begin{equation}
  \label{eq:appconst}
  c(\xi)=\liminf\{|\xi\q-\p|_2\,|\q|= 
|(\xi-\pq)\,\q^2|_2\,\colon \pq\in \hurat\}   \leqslant \sqrt{2/5}
\end{equation}
was established by Speiser~\cite{Speiser1932} for asymptotic
approximation, \ie, for all $\xi\in \quat$, there exist infinitely
many pairs $\p, \q\in \hur$ such that
\begin{equation}
    \label{eq:QDio}
    \left\vert \xi- \p\,\q^{-1}\right\vert_2  
    \leqslant \sqrt{\frac{2}{5}} \dfrac{1}{\vert \q\vert_2^2}
< \dfrac{1}{\vert \q\vert_2^2}. 
  \end{equation} 
  A.~L.~Schmidt~\cite{ASchmidt69} showed that $\sqrt{2/5}$ could not
  be reduced, so that it is analogous to Hurwitz's best possible
  rational approximation constant $1/\sqrt{5}$ for real
  numbers~\cite[\S11.8]{HW} and Ford's $1/\sqrt3$ for complex
  numbers~\cite{Ford1925} (see also~\cite{Vulakh97}).  Since the
  approximating rational quaternions are not required to be in their
  lowest terms here, the inequality~\eqref{eq:QDio} holds for all
  $\xi\in \quat$.


\subsection{Badly approximable quaternions}
\label{subsec:QBA} 

In parallel with the classical case, Dirichlet's theorem for
quaternions is best possible in the sense that the exponent 2
in~\eqref{eq:QDio2} is best possible. 
Accordingly, a quaternion $\xi$ for which there exists a constant
$c>0$ such that
\begin{equation*}
  \label{eq:QBA}
|\xi-\pq|\geqslant \frac{c}{|\q|_2^2}
\end{equation*}
for all $\pq$ is called {\em{badly
    approximable}}. By~\eqref{eq:QDio}, $c \leqslant
\sqrt{2/5}$.  
 The set of badly
approximable approximable quaternions will be denoted $\fBQ$.

\section{$\Psi$-approximable quaternions}
\label{sec:Psiapproxquats}

The inequality~\eqref{eq:QDio} establishes that
there are infinitely many $\pb$, $\q^{-1}$ in $\hur$ such that 
the approximants $\pq\in \hurat$ of Euclidean
distance are at most $|\q|_2^{-2}$ from the quaternion $\xi\in \Delta$.
As in the real case, it is natural to replace the error by a general
approximation function $\Psi$, \ie, a function $\Phi\colon
(0,\infty)\to (0,\infty)$ such that $\Psi(x)\to 0$ as $x\to\infty$. We
then consider the general inequality
\begin{equation}
  \label{eq:qda}
  |\xi-\pb\,\q^{-1}|_2<\Psi(|\q|_2)
\end{equation} 
for $\xi\in\quat$, or without loss of generality, for $\xi$ in the
compact set $\overline \Delta$ with $\Delta$ the $\cH$-fundamental
region from (\ref{eq:Delta}).
The Euclidean norm $|\q|_2$ chosen for quaternions and the argument
$|\q|_2$ of
the approximation function $\Psi$ being defined on 
$\sqrt{\N}=\{\sqrt{k}\colon k\in \N\}$. (This minor
complication would be avoided by working with the square of the norm
but then the analogy with $\R$ would not be so close.)  To make life
simpler and to make comparison with other types of \DA{} easier, we
will take $\Psi$ to be a step function satisfying
\begin{equation*}
  \label{eq:Psi[x]}
  \Psi(x)=\Psi([x]),
\end{equation*}
where $[x]$ is the integer part of $x$.

The main objective of this paper is to determine the metrical
structure of the set
\begin{equation*}
  \label{eq:cW} 
\cW(\Psi)=
\left\{\xi\in\Qu\colon |\xi-\pb\,\q^{-1}|_2<\Psi(|\q|_2) \
{\text{ for infinitely many }} \pb,\q\in\hur\right\}
\end{equation*}
and some related sets. Choosing the approximation function $\Psi$ as
above is natural and fits in with a Duffin-Schaeffer
conjecture~\cite[pg.~17]{Sprindzuk} for quaternions that we will not
address here. Nevertheless, the conjecture is still problematic as a
more appropriate choice of argument for $\Psi$ would be the Hurwitz
integer $\q$ rather than an integer $k$ (see~\cite{HarmanMNT}).
 Restricting the approximating Hurwitz rationals $\pq$
    to those with $\p,\q$ coprime, \ie, to the subset
\begin{equation*}
  \label{eq:cW'}
  \cW{\,'}(\Psi)=
\left\{\xi\in\Qu\colon |\xi-\pb\,\q^{-1}|_2<\Psi(|\q|_2) \
{\text{ for infinitely many }} \pb,\q\in\hur, ( \pb,\q)_r=1\right\}
\end{equation*}
of $\cW(\Psi)$, raises some minor technicalities and will not be
considered.

\emph{Henceforth, unless otherwise stated, $\Psi\colon
\N\to (0,\infty)$ will be a (monotonic) decreasing approximation function}.

Theorem~\ref{thm:QDT} implies that if $\Psi(x)$ increases, then
$\cW(\Psi)=\Qu$.  Removing monotonicity altogether turns out to be a
difficult and subtle problem, associated with the Duffin-Schaeffer
conjecture.  Although we will be concerned mainly with monotonic
decreasing approximation functions, we could, without loss of
generality, take $\Psi$ to be simply monotonic in some general
statements.  The union of translates by Hurwitz integers of the
compact subset
\begin{equation*}
  \label{eq:V01}
\cV(\Psi):=\cW(\Psi)\cap \overline\Delta=
\left\{\xi\in\overline\Delta\colon |\xi-\pb\,\q^{-1}|_2<\Psi(|\q|_2) \
{\text{ for infinitely many }} \pb,\q\in\hur\right\},
  \end{equation*}
of $\cW(\Psi)$ yields 
\begin{equation}
  \label{eq:WUV}
  \cW(\Psi)=\bigcup_{\pb\in \hur}(\cV(\Psi) + \pb). 
\end{equation}
Thus the measure of $\cW(\Psi)$ can be obtained from that of
$\cV(\Psi)$. The same holds for the set $\cV'(\Psi):=\cW'(\Psi)\cap
\overline\Delta$. 

\subsection{Resonant points, resonant sets and near-resonant sets}
\label{subsec:resnearesets}

Diophantine equations and approximation can be associated with the
physical phenomenon of resonance and for this reason the rationals
$p/q$ are referred to as {\em{resonant}} points in $\R$ (the
terminology is drawn from mechanics, see for
example~\cite[\S18]{ArnoldGM}).  From this point of view, the Hurwitz
rationals $\pq\in \hurat$ are resonant points in $\quat$.  In view
of~\eqref{eq:WUV}, there is no loss of generality in considering
quaternions restricted to $\overline \Delta$.  For each non-zero
$\q\in\hur$, the lattice $\cR_\qb$ of Hurwitz rationals or resonant
points $\pq$ in $\overline\Delta$ given by
\begin{equation*}
  \label{eq:quatreset}
  \cR_\qb =\left\{\pb\qb^{-1}\colon \pb\in\hur\right\}\cap \overline\Delta 
\end{equation*}
is useful in calculations.  This resonant set is an analogue in $\Qu$
of the set of equally spaced points $\{p/q\colon 0\leqslant p\leqslant q\}$ in
$[0,1]$.

For each $\q$, the number $\card \cR_\q$ of Hurwitz rationals $\pq$ in
$\overline \Delta$ is the number of $\p$ in $\q\overline\Delta$, \ie,
\begin{equation}
\label{eq:cardRq} 
\card \cR_\q  =
\kern-4mm\sum_{\p\in\hur\colon \pq\in \overline\Delta} \kern-4mm1 = 
|\q|_2^4 + O(|\q|_2^3) \asymp |\q|_2^4.
\end{equation}
The set $\cR:=\{\cR_\q\colon \q\in\hur\setminus \{0\}\}=\hurat\cap
\overline \Delta $ consists of the Hurwitz rationals $\hurat$ in
$\overline\Delta$.

Let $B_0:=B(\xi_0;r)=\left\{\xi\in \Qu\colon \left\vert
    \xi-\xi_0\right\vert_2 < r\right\}$ be the quaternionic ball
centred at $\xi_0$ with radius~$r$ and 4-volume
$|B_0|=\pi^2 r^4/2 \asymp r^4$~\eqref{eq:4ballmeasure}.  
The number of Hurwitz integers $\p$ in $N\,\overline\Delta$ is $N^4
+ O(N^3)$.  Thus by volume considerations, the number of resonant
points $\pq $ with $|\p|_2< |\q|_2$ satisfies
\begin{equation*}
 \label{eq:cardpq}
 \sum_{\p\in\hur\colon |\p|_2<|\q|_2} \kern-4mm1 
 = 2\frac{\pi^2}{2}|\q|_2^4 + O(|\q|_2^3)  
= \pi^2 |\q|_2^4 + O(|\q|_2^3)  \asymp |\q|_2^4 .
\end{equation*}
The number of resonant points $\pq$ with $\p,\q$ coprime could be
considered by using the quaternionic analogue of Euler's $\phi$ function
but this raises some complicated technicalities and will not be pursued here.

For each non-zero $\q\in\hur$, let 
\begin{equation*}
  \label{eq:Bcd}
  \cB(\cR_{\q};\varepsilon)=\bigcup_{\pb\in\hur} B(\pq,\varepsilon)
  \cap \overline \Delta=\left\{\xi\in \overline\Delta \colon
    \left\vert \xi-\pb\q^{-1}\right\vert <\varepsilon \text{ for some } \pb\in
    \hur\right\}
\end{equation*}
be the set of balls $B(\pq,\varepsilon)$ in $\overline \Delta$.  The
points in $\cB(\cR_{\q},\varepsilon)$ are within $\eps$ of a resonant
point and so will be called {\em{near-resonant}} points. The centres
$\pq$ lie in $ \cR_q$ and the number of such balls is $\asymp
|\q|^4$. Clearly $\cB(\cR_\q,\eps)$ is a finite lattice or array of
quaternionic balls $B(\pb\q^{-1},\varepsilon)\cap \overline \Delta$.
By~\eqref{eq:4ballmeasure} and~\eqref{eq:cardRq}, we have $ |
B(\pq,\varepsilon)|\ll \varepsilon^4$ and, provided $\eps$ is small
enough, the near-resonant set $\cB(\cR_\q,\eps)$ has Lebesgue measure
\begin{equation}
  \label{eq:|BRq|}
  |\cB(\cR_\q,\eps)|\asymp |\q|_2^4\,\eps^4.
\end{equation}

\subsection{The structure of $\cV(\Psi)$} 
\label{sec:VPsi}
It is readily verified that the set $\cV(\Psi)\subset
\overline\Delta$ can be expressed in the form of a `limsup set'
involving unions of near-resonant sets as follows:
\begin{eqnarray}
  \label{eq:limsupV}
   \cV(\Psi)&=& \bigcap_{N=1}^\infty \kern1mm \bigcup_{n=N}^\infty 
 \kern1mm \bigcup_{[\vert \q\vert_2]=n}
  \cB(\cR_\q,\Psi(\vert \q\vert_2)) = 
\bigcap_{N=1}^\infty\bigcup_{\vert \q\vert_2\geqslant N} 
 \cB(\cR_\q,\Psi(\vert \q\vert_2)) \notag\\
&:=&  \limsup_{\vert \q\vert_2 \to\infty}
  \cB(\cR_\q,\Psi(\vert \q\vert_2)).  
\end{eqnarray}
Similarly 
\begin{equation}
  \label{eq:limsupW}
  \cW(\Psi) = 
\bigcap_{N=1}^\infty\kern1mm\bigcup_{\vert \q\vert_2\geqslant N}\kern1mm 
\bigcup_{\p\in\hur}B(\pq,\Psi(\vert \q\vert_2)) 
  =  \limsup_{\vert \q\vert_2 \to\infty} \kern1mm 
  \bigcup_{\p\in\hur}
 B(\pq,\Psi(\vert \q\vert_2)).  
\end{equation}
It follows that $\cV(\Psi)$ has a natural cover
\begin{equation}
  \label{eq:cover}
  \mathcal{C}_N(\cV(\Psi))=  
  \{ \cB(\cR_\q,\Psi(\vert \q\vert_2))\colon \vert \q\vert_2 \geqslant N\}
\end{equation}
for each $N=1,2,\dots$.  By~\eqref{eq:|BRq|}, the Lebesgue measure
of $\cB(\cR_\q,\Psi(\vert \q\vert_2))$ satisfies
\begin{equation*}
  \label{eq:|BRqPsi|}
  |\cB(\cR_\q,\Psi(\vert \q\vert_2))| \asymp|\q|_2^4 \Psi(\vert
  \q\vert_2)^4 .
\end{equation*}

\subsection{Approximation involving a power law}
\label{subsec:powerlaw}
In the special case that $\Psi(x):=x^{-v}$, $(v>0)$, we write
$\cV(\Psi):=\cV_v$ and $\cW(\Psi):=\cW_v$.  When $v=2$, it follows
from their definitions (\eqref{eq:limsupV},~\eqref{eq:limsupW})
and from~\eqref{eq:QDio} that 
\begin{equation}
\label{eq:V2W2}
\cV_2=\limsup_{|\q|_2\to\infty} \cB(\cR_\q,\vert
  \q\vert_2^{-2})=\overline\Delta \ \ {\text{and}} \ \
\cW_2 = \limsup_{|\q|_2\to\infty} 
 \bigcup_{\p\in\hur}  B(\pq,\vert \q\vert_2^{-2}) 
  = \Qu.   
\end{equation}

It is evident that for $v'\geqslant v$, $ \cW_{v'}\subseteq \cW_v$ and
$\cV_{v'}\subseteq \cV_v$.  For $v>2$, $\cW_v$ will be called the set
of {\em{very well approximable}} quaternions.  Analogous definitions
can be made for $\R^n$ and other spaces.

\section{Metrical \DA\ in $\H$: the quaternionic Four Peaks}
\label{sec:4peaksinQ}

In order to provide a convenient comparison with the real case, the
analogous results for quaternions are now set out in the same order as
in~\S\ref{subsection:classicalmDA}.

\paragraph{The First Peak: Khintchine's theorem for $\Qu$.}
As in the real case, the quaternionic Khintchine's theorem relates the
Lebesgue measure of the set $\cW(\Psi)$ of $\Psi$-approximable
quaternions to the convergence or divergence of a certain `volume' sum
while the analogue for \Ja's extension of \K's theorem does the same
for Hausdorff $f$-measure. The quaternionic version of \K's theorem is
now stated.

\begin{theorem} 
 \label{thm:qKT} 
 Let $\Psi\colon \N\to (0,\infty)$. Then the sets
  \begin{equation*}
    \cW(\Psi) \ and \ \cV(\Psi) \text{ are } 
\begin{cases}
      {\text null} & \text{ when } \sum_{m=1}^\infty \Psi(m)^4 \, m^7<\infty, \\
      {\text full}  & \text{ when $\Psi$ is decreasing and }  
      \sum_{m=1}^\infty\Psi(m)^4 \, m^7 =\infty. \\
    \end{cases}   
  \end{equation*}
  \end{theorem}
  Note that when $\cV(\Psi)$ has full Lebesgue measure,
  $|\cV(\Psi)|=|\Delta|=1/2$.  Again, it is evident that the value of
the critical `volume' or `measure' sum 
\begin{equation}
  \label{eq:critsumLq}
\sum_{m=1}^\infty \Psi(m)^4 m^7  
\end{equation}
determines the Lebesgue measure of $\cW(\Psi)$ and $\cV(\Psi)$.
Similar critical sums are associated with Hausdorff measures.

\paragraph{The Second Peak: \Ja's Hausdorff measure theorem for $\quat$.}

\begin{theorem}
  \label{thm:qJT} 
  Let $f$ be a dimension function with $f(x)/x^4$ decreasing and
  $f(x)/x^4\to\infty$ as $x\to 0$.  Then
  \begin{equation*}
    \sH^f(\cW(\Psi))=\sH^f(\cV(\Psi))=
    \begin{cases}
      0 & \text{ when }     \sum_{m=1}^\infty m^7 f(\Psi(m)) <\infty, \\
      \infty  &  \text{ when $\Psi$ is decreasing and }
\sum_{m=1}^\infty m^7 f(\Psi(m)) =\infty.
    \end{cases}
  \end{equation*} 
\end{theorem}
The sum 
\begin{equation}
  \label{eq:Hcritsum}
  \sum_{m=1}^\infty m^7 f(\Psi(m))
\end{equation}
is the critical sum for Hausdorff $f$-measure. This $f$-measure version of
Theorem~\ref{thm:qKT} does not hold for Lebesgue measure but the two
theorems can be combined into a single quaternionic `\K-\Ja' result
(see~\cite[\S2.3]{BVParis2009a}).
\begin{theorem}
  \label{thm:qKJT} 
Let $f$ be a dimension function with $f(x)/x^4$ decreasing. Then 
  \begin{equation*}
    \sH^f(\cV(\Psi))=
    \begin{cases}
      0 & \text{ when }     \sum_{m=1}^\infty m^7 f(\Psi(m)) <\infty, \\
      \sH^f(\overline\Delta) &  \text{ when  $\Psi$ is decreasing and }
\sum_{m=1}^\infty m^7 f(\Psi(m)) =\infty.
    \end{cases}
  \end{equation*} 
\end{theorem}
The Mass Transference Principle (see~\S\ref{subsec:mtp2} below) can
also be used to deduce this theorem from Theorem~\ref{thm:qKT}.
Specialising Theorem~\ref{thm:qJT} to Hausdorff $s$-measure, where
$f(x)=x^s$, gives
\begin{theorem}
  \label{thm:qJsm}
  Suppose $0\leqslant s <4$.  Then
  \begin{equation*}
    \sH^s(\cW(\Psi))=\sH^s(\cV(\Psi))=
    \begin{cases}
    0, & \text{ when }  \sum_{m=1}^\infty m^7\Psi(m)^s <\infty, \\
     \infty,
&  \text{ when $\Psi$ is decreasing and }
\sum_{m=1}^\infty m^7\Psi(m)^s  =\infty.
    \end{cases}
  \end{equation*} 
\end{theorem} 
\noindent   Specialising further  to the Hausdorff $s$-measure for a power law
approximation function, \ie, to $\Psi(m)=m^{-v}$, where $v>0$, gives
\begin{theorem}
  \label{thm:qsm} 
Suppose $v>2$. Then
    \begin{equation*}
    \sH^s(\cW_v)= \sH^s(\cV_v)=
    \begin{cases}
      0 & \text{ when }   s > 8/v, \\
      \infty  & \text{ when } s \leqslant 8/v.
    \end{cases}
  \end{equation*} 
\end{theorem} 

\paragraph{The Third Peak: the \JB{} theorem for $\quat$.} The \HD{}
of $\cW_v$ is the point of discontiunuity of the Hausdorff measure
$\sH^s(\cW_v)$ and so the quaternionic version of the \JB{} theorem
follows by definition from the above result.
 \begin{theorem}
\label{thm:qJBT}
 Let $v\geqslant 0$. Then the \HD{} of $\cW_v$ is given by
 \begin{equation*}
    \hdim \cW_v=\hdim \cV_v=
    \begin{cases}
      4  & \text{ when } v\leqslant 2, \\
      \dfrac{8}{v} & \text{ when } v > 2.
    \end{cases}
  \end{equation*}
\end{theorem}
Note that $\sH^s(\cW_v)=\infty$ when $s=\hdim \cW_v=8/v$.  A 
proof of this result will also be given in~\S\ref{subsec:appnpowerlaw}
below, using the Mass Transference Principle (see~\S\ref{subsec:mtp2}
below) and the quaternionic Dirichlet theorem
(Theorem~\ref{thm:QDT}).

\paragraph{The Fourth Peak: \Ja's theorem for ${\mathfrak B}_{\mathbb{H}}$.}  The
definition of $\QBA$, the set of badly approximable quaternions, is
given in~\S\ref{subsec:QBA} above.

  \begin{theorem} 
\label{thm:QJTBA}
The set $\QBA$ is null with full \HD, \ie,    $|\QBA|=0 $ and 
 \begin{equation*}
       \hdim ({\mathfrak B}_{\mathbb{H}}) = 4.
    \end{equation*}
  \end{theorem}

\section{Ubiquitous systems }
\label{sec:ubiquity}

As has been pointed out in~\S\ref{sec:QDA}, the metrical structure of
lim sup sets which arise in number theory and elsewhere can be
analysed very effectively using ubiquity.  A ubiquitous system (or
more simply ubiquity) is a more quantitative form of density
underlying the classical Lebesgue and the more delicate Hausdorff
measure results.  Originally introduced to investigate lower bounds
for Hausdorff dimension~\cite{DRV90a}, ubiquitous systems have been
extended considerably and now provide a way of determining the
Lebesgue and Hausdorff measures of a very general class of `limsup'
sets~\cite[Theorems 1 \& 2]{BDV06}.  Indeed using the Mass
Transference Principle, these two measures have been shown to be
equivalent for this class of limsup sets, rather than Hausdorff
measure being a refinement of Lebesgue~\cite{BVParis2009a}.

\subsection{A metric space setting}
\label{subsec:ms}

The definition of ubiquity given in~\cite{BDV06} applies to a compact
metric space $(\Omega,d)$ with a non-atomic finite measure $\mu$ (which
includes $n$-dimensional Lebesgue measure). The resonant sets play the
role of the approximants, which in the real line consist of the
rationals. We will give a simplified version appropriate for \DA{} in
$\quat$. The deep arguments in~\cite{BDV06,ktv06} are based on dyadic
dissection suited to the Cantor-type constructions used in the
proof. Thus the important ubiquity sum~\eqref{eq:qubiksum} is
2-adic, unlike the critical sum~\eqref{eq:critsumLq} which emerges
from simpler standard estimates. 

We start with a family $\cR$ of {\it{resonant sets}}  $R_j$ in $\Omega$,
where $j$ lies in a countable discrete index set $J$ with
each $j\in J$ having a weight $\lfloor j\rfloor$.  The number of $j$
satisfying $\lfloor j\rfloor \leqslant N$ is assumed to be finite for each
$N\in\N$. In $\quat$ we take 
$j=\q$ , the index set $J=\{\q\in \hur \colon \q\ne 0\}$,
and the weight $\lfloor j\rfloor=\lfloor\q\rfloor := |\q|_2$. 
The resonant set $R_j=\cRq$ corresponds to the lattice $\cRq$ of
resonant points $\pq\in\hur$ or in $\hur\cap \overline \Delta$.
 In the general formulation, the resonant sets $R_j$
 can be lines, planes etc.

Let $B_0:=B(\xi_0,r)=\{\xi\in \Omega\colon d(\xi,\xi_0)<r\}$, for $r>0$,
be any fixed ball in $\Omega$ and let $\cR$ be the family $\{R_j\colon
j\in J\}$ of resonant points in $\Omega$.  Further, let $\Psi$ be an
approximation function, \ie, $\Psi\colon (0,\infty)\to (0,\infty)$  
converges to 0 at $\infty$.  Let $\rho\colon \N\to
(0,\infty)$ be a function with $\rho(m)=o(1)$. 
If for a given $B_0$,
\begin{equation}
  \label{eq:ubidef}
  \mu(B_0\cap\kern-2mm\bigcup_{1 \leqslant \lfloor j\rfloor \leqslant N} 
\kern-2mmB(R_j,\rho(N)))\gg \mu(B_0),
\end{equation}
where the implied constant in~\eqref{eq:ubidef} is independent of $B_0$,
then the family $\cR=\{R_j\colon j\in J \}$ is said to be a
{\em{(strongly) ubiquitous system with respect to the function}}
$\rho$ and the weight $\lfloor\,\cdot\,\rfloor$. The idea here is that
the family of near-resonant balls $ B(R_j,\rho(N))$ meets the
arbitrary ball $B_0$ in $\Omega$ substantially and covers it at least
partially in measure.  This can be regarded as a fairly general
Dirichlet-type condition in which a `significant' proportion of points
is close to some resonant point $R_j$.  It is evident that we want
$\rho$ as small as possible.  Note that in~\cite{DRV90a}, $\rho$ was
required to be decreasing; this condition is no longer required in the
improved formulation in~\cite{BDV06}.  In applications, $\rho$ can
often be chosen to be essentially a simple function, such as a
power. In particular, for quaternionic \DA, the choice of exponent is
2 (see~\S\ref{subsec:eta}).  This is the same exponent as the ubiquity
function for rational approximation on the real line $\R$ and is quite
different from that for simultaneous rational approximation
(see~\S\ref{subsec:sdaR4}).  The reason goes back to the similarity
between the Dirichlet's theorems for the two spaces.

The set of points in $\Omega$ which are $\Psi$-approximable by the
family $\cR=\{\cR_j\colon j\in J\}$ with respect to the weight
$\lfloor \cdot \rfloor$ is defined by
\begin{equation}
\label{eq:Lambda}
  \Lambda(\Psi): = \{\xi\in\Omega \colon \xi\in B(R_j,\Psi(\lfloor j\rfloor))
{\text{ for infinitely many }} j\in J\}. 
\end{equation}
If the family $\cR$ is a ubiquitous system with respect to a suitable 
 $\rho$ and weight, then the metrical structure of
$\Lambda(\Psi)$ can be determined.  Note that the set on the right hand
side of
(\ref{eq:Lambda}) can be rewritten as a `limsup' set (and hence falls
into the ambit of the framework in~\cite{BDV06}) as follows,
\begin{equation*}
  \label{eq:limsupL}
   \Lambda(\Psi)=\bigcap_{N=1}^\infty \bigcup_{m=N}^\infty
\bigcup_{\lfloor j\rfloor] = m} B(R_j,\Psi(\lfloor j\rfloor))
= \bigcup_{\lfloor j\rfloor\geqslant N} 
B(R_j,\Psi(\lfloor j\rfloor)) =\limsup_{\lfloor j\rfloor\to\infty} 
B(R_j,\Psi(\lfloor j\rfloor)).
\end{equation*}
Thus for each $N=1,2,\dots$, we have
\begin{equation*}
  \label{eq:natcover}
  \Lambda(\Psi)\subseteq 
\bigcup_{\lfloor j\rfloor\geqslant N} 
B(R_j,\Psi(\lfloor j\rfloor))
=\cC_N,
\end{equation*}
where $\cC_N=\{B(R_j,\Psi(\lfloor j\rfloor))\colon \lfloor j\rfloor
\geqslant N \}$ is the natural cover for $\Lambda(\Psi)$;
the cover for $\cV(\Psi)$ given in~\eqref{eq:cover} is a special case.

\section{ The proof of \K's theorem for $\H$ when the
  critical sum converges}
\label{sec:Kconvergent}

The straightforward proof follows from~\eqref{eq:limsupV} and the form
of the natural cover $\cC_N(\cV(\Psi))$ for
$\cV(\Psi)$~\eqref{eq:cover}.  It follows by~\eqref{eq:|BRq|} that for
each $N=1,2,\dots$, the Lebesgue measure of $\cV(\Psi)$ satisfies
\begin{equation} 
\label{eq:uppbd|WPsi|1}
  \vert \cV(\Psi) \vert \leqslant
  \sum_{m=N}^\infty 
  \sum_{ m\leqslant |\q|_2 < m+1} 
\kern-2mm\vert B(\cR_\q,\Psi(\vert \q\vert_2))\vert \ll 
  \sum_{m=N}^\infty
  \sum_{ m \leqslant |\q|_2 < m+1}  \kern-2mm|\q|_2^4 \ \Psi(|\q|_2)^4
\end{equation} 
By~\cite[Th.~386]{HW}, the number $r_4(m)$ of Hurwitz integers $\q$
with $|\q|_2^2=m$ is given by 
\begin{equation*}
  \label{eq:r4m}
  r_4(m) =  8\kern-2mm\sum_{d|m,\, 4 \not \; | \, d} \kern-2mmd
\end{equation*}
but for our purposes a simpler estimate suffices.  By volume
considerations,
\begin{equation*}
   \sum_{|\q|_2 < m+1} \kern-3mm1 = \frac{\pi^2}{2}\,  2\, m^4+O(m^{3}) 
\sim   \pi^2 m^4 ,
\end{equation*}
whence for each $m\in\N$, 
\begin{equation*}
  \sum_{m  \leqslant |\q|_2 < m+1} \kern-3mm1 =  
\kern-3mm\sum_{|\q|_2 < m+1} \kern-3mm1 - \kern-2mm\sum_{|\q|_2 < m} \kern-2mm1   
  \ll m^3,
\end{equation*}
where in the sum on the left hand side, $|\q|_2$ ranges over the $2m+1$ values
$$
m,\sqrt{m^2+1},\dots, \sqrt{(m+1)^2-1}.
$$ 
But $\Psi(|\q|_2) := \Psi([|\q|_2]) =\Psi(m)$ when $ m \leqslant
|\q|_2 < m+1$, so that 
\begin{equation}
\label{eq:uppbd|WPsi|2}
  \vert \cV(\Psi) \vert  
\ll  \sum_{m=N}^\infty \Psi(m)^4  (m+1)^4 
\kern-3mm\sum_{m \leqslant |\q|_2 < m+1} \kern-3mm1
\ll\sum_{m=N}^\infty \Psi(m)^4  m^7. 
\end{equation} 
Thus for each $N=1,2,\dots$, the measure of $\cV(\Psi)$ satisfies
\begin{equation*} 
\label{eq:upperbd|W*|}
  \vert \cV(\Psi) \vert  \ll
  \sum_{m=N}^\infty \Psi(m)^4 m^7 . 
\end{equation*}
Since $N$ is
arbitrary, if the critical sum~\eqref{eq:critsumLq} converges then
the tail
$\sum_{m=N}^\infty \Psi(m)^4 m^7 $ converges to 0 and
$\cV(\Psi) $ is a null set, \ie, 
\begin{displaymath}
  |\cV(\Psi)|=|\cW(\Psi)|=0. 
\end{displaymath}
This is the convergence part of Theorem~\ref{thm:qKT}, the
quaternionic analogue of Khintchine's theorem.  Note that since
$\cV\,'(\Psi)\subset \cV(\Psi)$, the convergence of the critical sum
implies that $|\cV'(\Psi)|=|\cW\, '(\Psi)|=0$ also. 

\section{The proof when the critical sum diverges}
\label{sec:divergentsum}
  The  case of divergence is much more difficult.  
 The ideas involved, particularly ubiquity
(see~\S\ref{sec:ubiquity}) and the remarkable 
Mass Transference Principle 
(see~\S\ref{subsec:mtp2}), require some further definitions and
notation. 
First we explain the  principle in a simple setting
to clarify the ideas and give an application to indicate its power.
Then we explain ubiquity.

\subsection{The Mass Transference Principle}
\label{subsec:mtp2}

The Mass Transference Principle, introduced by Beresnevich and Velani
in~\cite{BV06}, is a remarkable technique which allows Lebesgue
measure results for lim sup sets to be transferred to Hausdorff
measures.  A version adapted to our purposes is now given.  Let $n$ be
a non-negative integer and let $f$ be a dimension function
(see~\S\ref{sec:MaD}) such that $x^{-n}f(x)$ is
monotonic. For any ball $B=B(c,r)$ centred at $c$ and radius
$r$, let
\begin{equation*}
  \label{eq:6}
  B^f:=B(c,f(r)^{1/n}).
\end{equation*}
When  for $s>0$, $f(x)=x^{-s}$, write $B^f=B^s$, so that
$B^s(c,r)=B(c,r^{s/n})$; note that $B^n=B$.  Similarly for a
family~$\cB=\{B(c_i,r_i)\}$ of balls in $\Omega$, let
\begin{equation*}
  \label{eq:family}
  {\cB}^f:=\{B(c_i,f(r_i)^{1/n})\},
\end{equation*}
so that $\cB^n= \cB$.  Let $\{\cB_i=\cup_j B(c_{i_j},r_{i})\colon
i\in\N\}$ be a family of finite unions of balls $ B(c_{i_j},r_{i})$ in
$\R^n$ with the same radius $r_i\to 0$ as $i\to\infty$.  Suppose that
for any ball $B_0\in \R^n$,
\begin{equation*}
  \label{eq:MTP1}
  |(B_0\cap\limsup_{i\to\infty} \cB_i^f)|= |B_0|.
\end{equation*} 
Then the  Mass Transference Principle asserts that 
\begin{equation*}
  \label{eq:8}
  \sH^f(B_0\cap \limsup_{i\to\infty} \cB_i)=\sH^f(B_0).
\end{equation*}
Thus the appropriate version of Khintchine's theorem would imply \Ja's
$f$-measure theorem.  This has not been proved for quaternions but~\eqref{eq:appconst} 
can be used with the mass transference principle to prove
Theorem~\ref{thm:qJBT}, the quaternionic analogue of the \JB{}
theorem.

\subsection{An application to $\cV_v$\,: the quaternionic \JB{} theorem} 
\label{subsec:appnpowerlaw}
In \S\ref{subsec:mtp2}, take $n=4$, $f(x)=x^s$, $s < 4$ and
$\Psi(x)=x^{-v}$, $v>0$, $c_i=\pq$ (recall that $\p,\q$
  are not necessarily coprime) and $r_i=\Psi(|\q_2|)$.  Let the set
$\cB_\q:=\cB(\cR_\q,|\q|_2^{-v})$ correspond to the set $\cB_i$
in~\S\ref{subsec:mtp2} and $ \cB_\q^{8/v}=\cB(\cR_\q,|\q|_2^{-2}) $
correspond to $\cB_i^{8/v}$.

Suppose $v\leqslant 2$. Then by~\eqref{eq:V2W2}, 
\begin{equation*}
 \cV_v= \limsup_{|\q_2|\to\infty} \cB(\cR_\q,|\q|_2^{-v}) =
 \overline\Delta, 
\end{equation*}
whence {\it a fortiori}, $|\cV_v|=
|\overline\Delta| = 1/2 
$
and $ \dim \cV_v= 4$.

Suppose $v>2$. By the definition of $\cB_\q^{8/v}$ and
by~\eqref{eq:QDio}, 
 \begin{equation*}
   \limsup_{|\q|_{_2}\to\infty}  \cB_\q^{8/v} = \limsup_{|\q|_{_2}\to\infty}
\cB(\cR_\q,|\q|_2^{-2}) = \overline\Delta,
 \end{equation*}  
whence  
 \begin{equation*}
  \left|B_0\cap \limsup_{|\q|_{_2}\to\infty}  \cB_\q^{8/v}\right| = 
\left|B_0\cap\overline\Delta\right|=\left|B_0\right|.
 \end{equation*}  
 It follows by the Mass Transference Principle that 
\begin{equation*}
 \sH^{8/v}\left(B_0\cap \limsup_{|\q|_{_2}\to\infty} \cB_\q \right) =
 \sH^{8/v}\left(B_0\cap \cV_v\right) =
 \sH^{8/v}(B_0)=\infty
\end{equation*}
since $B_0$ is open  and $8/v<4$.  But $ B_0\cap
\cV_v \subset \cV_v$, whence for $s\leqslant 8/v$,
\begin{equation*}
  \sH^s(\cV_v)=  \sH^{8/v}(\cV_v)=\infty
\end{equation*}
 and from the definition of \HD,  $\hdim \cV_v \geqslant 8/v$.
\vspace{0.05in}

Next suppose $s>8/v$.  By~\eqref{eq:cover}, for each $N=1,2,\dots$,
the family of balls
\begin{equation*}
  \{ B(\pq,\vert \q\vert_2^{-2})  
\colon |\p|_2\leqslant |\q|_2, \vert \q\vert_2 \geqslant N\}
\end{equation*}
is a cover for $\cV_v$.  Hence by~\eqref{eq:Hfmeasure}, for each
$N\in\N$,
\begin{eqnarray*}
 \sH^s(\cV_v)&\leqslant&
\sum_{m=N}^\infty \sum_{m \leqslant|\q|_2 < m+1} 
\sum_{\p\colon \p\in \overline\Delta\,\q}
\kern-2mm\left(\di B(\pq,\vert \q\vert_2^{-v})\right)^{s} 
\\ 
&\ll&  \sum_{m=N}^\infty  \sum_{m \leqslant |\q|_2 < m+1} 
\sum_{\p\colon \p\in \overline\Delta\,\q} 
\kern-2mm|\q|_2^{-sv}
\ll  \sum_{m=N}^\infty \sum_{m \leqslant |\q|_2 < m+1}
  \kern-3mm|\q|_2^4 |\q|_2^{-sv} \\
&\ll& \sum_{m=N}^\infty m^{4-sv} \kern-3mm\sum_{m \leqslant |\q|_2 <
  m+1}  \kern-4mm1 
\ll  \sum_{m=N}^\infty m^{7-vs} \to 0 \ {\text{as}} \ N \to \infty,
\end{eqnarray*}
since $s>8/v$. Thus $\sH^s(\cV_v)=0$ for $s>8/v$ and $\hdim \cV_v \leqslant
8/v$. Combining the values of  $\sH^s(\cV_v)$ gives for 
$v>2$,
\begin{equation*}
    \sH^s( \cW_v)= 
    \sH^s( \cV_v)= 
    \begin{cases}
      \infty & \text{ when } s\leqslant \frac{8}{v}, \\
      0 & \text{ when } s>\frac{8}{v},
    \end{cases}
\end{equation*}
which is  Theorem~\ref{thm:qJT},  which in turn implies 
Theorem~\ref{thm:qJBT}, the quaternionic \JB{} theorem. 
Note that the Hausdorff $s$-measure is infinite at $s=\hdim \cV_v$.

Since $\Psi$ is decreasing and $f$ increasing, the composition 
  $f\circ\Psi$ is decreasing. Thus Khintchine's Theorem for
  quaternions (Theorem~\ref{thm:qKT}) implies that when the sum
  $\sum_m f(\Psi(m)) m^7$ diverges,
\begin{equation*}
  |B_0\cap (\limsup_{|\q|_2\to \infty}
\cB(\cR_\q,f(\Psi(|\q|_2))))| = |B_0|. 
\end{equation*}
Hence  by the Mass Transfer Principle, 
\begin{equation*}
  \sH^f(B_0\cap \cV(\Psi) ) = \sH^f(\overline\Delta),
\end{equation*}
so that divergent case of Khintchine's theorem implies that of \Ja's
$f$-measure theorem.  However, this case of Khintchine's
theorem needs to be proved and more ideas are needed to deal with the
general decreasing approximation function $\Psi$.

\subsection{The  quaternionic \K{} theorem in the divergent case}

The objective here is to complete the determination of the Lebesgue
and Hausdorff measures of $\Psi$-approximable quaternions when the
critical sums diverge.  We recall that the quaternions $\quat$ form a
4-dimensional metric space which naturally carries Lebesgue measure.
It is convenient to work with the compact set $[0,1]^3\times [0,1/2]
=\overline\Delta$, given in~\S\ref{subsec:resnearesets} above, for
$\Omega$ and with the set of $\Psi$-approximable quaternions in
$\overline\Delta$, \ie, with $\cV(\Psi)=\cW(\Psi)\cap
\overline\Delta$ instead of with $\cW(\Psi)$.

We begin by stating a simplified version of the Beresnevich-Velani
theorem~\cite[Th.~3]{BVParis2009a} for the ubiquitous systems
described in~\S\ref{sec:ubiquity}, and then deduce the analogue of
\K's theorem for $\quat$ in the divergence case. The
Beresnevich-Velani theorem holds for a compact metric space with a
measure comparable to Lebesgue measure. The theorem can be regarded as
a general \K-\Ja{} result and illustrates the power of ubiquity and mass
transfer (see~\S\ref{subsec:mtp2} above). Note  that in addition
to converging to 0 at infinity, the ubiquity function $\rho$ must also
satisfy the technical condition that  for some positive constant
$c<1$,
 \begin{equation}
  \label{eq:ubi2decay}
  \rho(2^{r+1})\leqslant c \rho(2^r)
\end{equation} 
for $r$ sufficiently large.
Such functions will be called {\em{dyadically decaying}}, a condition
which is satisfied in the applications considered here.  This
condition is weaker than the requirement in earlier work (see for
example~\cite{DRV90a}) that $\rho$ be decreasing. Note that the
definition in~\cite{BDV06} is more general: $\rho$ is `$u$-regular', a
condition which involves a sequence $(u_n\colon n\in\N)$.  The very
general Theorem~2 in~\cite{BDV06} could also be used to first prove
the analogue of \K's theorem and then the analogue of the \KJ{}
theorem deduced via mass transference~\S\ref{subsec:mtp2}.  
The dyadic decay condition can be imposed on $\Psi$ instead.

\begin{theorem}[Beresnevich-Velani]
\label{thm:BV}
Let $(\Omega,d)$ be a compact metric space equipped with a Borel
measure $\mu$ which for some $\delta>0$ satisfies 
\begin{equation}
\label{eq:mBrd}
 \mu(B(\xi,r))\asymp r^{\delta}
  \end{equation}
  for any sufficiently small ball $B(\xi,r)$ in $\Omega$.  Suppose
  that the family $\cR$ of resonant sets in $\Omega$ is a strongly
  $\mu$-ubiquitous system relative to the dyadically decaying function
  $\rho$ and that $\Psi$ is a decreasing approximation function.  Let
  $f$ be a dimension function with $f(x)/x^{\delta}$ monotonic.
If for some $\kappa>1$, the ubiquity sum
   \begin{equation}
      \label{eq:critubisum1}
    \sum_{m=1}^\infty
     \frac{f(\Psi(\kappa^m))}{\rho(\kappa^m)^{\delta}}
   \end{equation}
diverges, then the  Hausdorff $f$-measure $\sH^f(\Lambda(\Psi))$
is given by 
 \begin{equation*}
   \label{eq:mumeas1}
   \sH^f(\Lambda(\Psi))=\sH^f(\Omega).
 \end{equation*}
\end{theorem}

The hypotheses of Theorem~\ref{thm:BV} imply that $\mu$ is comparable
to the $\delta$-dimensional Hausdorff 
measure $\sH^\delta$ and that
$\dim \Omega=\delta$.  Note that in the \BVt, the
sum~\eqref{eq:critubisum1} is `$\kappa$-adic', whereas we have been
working with `standard' sums such as~\eqref{eq:critsumLq}.  By the
choice of $\rho$ we will make and by
Lemmas~\ref{lem:eta} and~\ref{lem:rH}, the
critical sum~\eqref{eq:critsumLq} will be comparable to the ubiquity
sum~\eqref{eq:critubisum1}.

\subsection{Perturbing divergent sums}
\label{subsec:eta}
The following lemma, drawn from~\cite{Casselshort}, is needed to
construct the ubiquity function $\rho$.

\begin{lemma}
\label{lem:eta}
  Let $F\colon \N \to (0,\infty)$ satisfy $\sum_{m=1}^\infty
  F(m)=\infty$.  Then there exists a  decreasing function $\eta\colon
  \N\to [0,1]$ with
  $\eta(m)=o(1)$, such that 
for any $\alpha>0$, the sequence $m\eta(m)^\alpha\to \infty$ as
$m\to\infty$,  $\eta(2^r)\leqslant 2 \eta(2^{r+1})$ 
and  such that $ \sum_{m=1}^\infty F(m)\,\eta(m) = \infty$, $r=1,2,\dots$.

\end{lemma}
\begin{proof}
  Since $\sum_{m=1}^\infty F(m)=\infty$, we can choose a strictly
  increasing sequence $(m_i\colon i=1,2,\dots)$ with $m_1=1$ such that 
$m_{i+1}\geqslant 2m_i\geqslant \dots\geqslant 2^i$ and 
  \begin{equation*}
    \sum_{m_{i} \leqslant m < m_{i+1}} F(m) > 1 .
  \end{equation*}
Define $\eta\colon \N\to [0,1]$ by
\begin{equation}
\label{eq:eta}
  \eta(m)=i^{-1}, \ m\in [m_{i}, m_{i+1}). 
\end{equation}
Evidently  $\eta$ is   decreasing, $o(1)$ and
\begin{equation*}
  m\,\eta(m)^\alpha \geqslant \frac{m_i}{i^\alpha} \geqslant 2^i \,i^{-\alpha} \to \infty
\end{equation*}
as $i$ and hence $m\to\infty$.  In addition, if $\eta(2^r)= 1/i$, then
by the choice of the intervals $[m_1,m_{i+1})$, 
  $\eta(2^{r+1})= 1/i$ or $1/(i+1)$, whence
  $ \eta(2^{r+1})\leqslant \eta(2^r)\leqslant 2 \eta(2^{r+1})$. 

Moreover  
\begin{equation*}
\sum_{m=1}^\infty F(m)\eta(m) =   
\sum_{i=1}^\infty \,\sum_{m_{i} \leqslant m < m_{i+1}} \kern-3mm\eta(m)  F(m)
=  \sum_{i=1}^\infty i^{-1} \kern-3mm\sum_{m_{i} \leqslant m < m_{i+1}}  \kern-3mmF(m) >
\sum_{i=1}^\infty i^{-1} =\infty.
\end{equation*}
\end{proof}
Thus $F$ can be replaced by a smaller function $F\eta$ without
affecting the divergence of the sum. Clearly $\eta$ depends on $F$.

\subsection{The functions $\eH$ and $\rh$}
\label{subsec:etaH}

Let $\eH=\eH(F)$ be the function in Lemma~\ref{lem:eta} corresponding
to $F(m)=f(\Psi(m))\,m^7$; recall that in the divergent case $
\sum_mf(\Psi(m))m^7=\infty$ by hypothesis. Define the function $\rh
\colon \N\to(0,1]$ by
\begin{equation}
  \label{eq:rH}
    \rh(m):=\frac{2}{\eH(m)^{1/4}m^2}\,.
\end{equation}
Then the function $\rh$ is the product of an inverse
  square and the slowly increasing function $1/\eH$. It turns out
that $\eH$ decreases sufficiently slowly to ensure that $\rh(m') \leqslant
2^{1/4} \rh(m)$ for $m'\geqslant m$ (so that $\rh$ is decreasing modulo
$2^{1/4}$).
\begin{lemma}
\label{lem:rH}
The function $\rh$ satisfies
\begin{enumerate}
\item   $\rh(m)= o(1)$,   
\item  $\rh(m)^{-1} = o(m^2)$,
\item   $\rh(m') \ll \rh(m)$ for all $m'\geqslant m$,
\item    $\rh$ decays dyadically.
\end{enumerate}
\end{lemma}

\begin{proof} 
 \noindent
 \begin{enumerate}
  \item 
  By Lemma~\ref{lem:eta}, $\eH(m)^{1/4} m \to \infty$ as $m\to\infty$,
  so $\rh(m):= 2(\eH(m)^{1/4}\,m^2)^{-1} \to 0$.
\item  Since $\rh(m):= 2(\eH(m)^{1/4}\,m^2)^{-1}$, we have that
$\rh(m)^{-1}m^{-2}=\eH(m)^{1/4}/2 \to 0$ as  $m\to \infty$.
   \item  Suppose $m\leqslant m'$.  We consider cases; recall $i\in\N$
   and that $m_{i+1}\geqslant 2m_i$. 
  When $m\leqslant m'$ and $m,m'\in [m_i,m_{i+1})$,
\begin{equation*} 
   \rh(m')=\frac2{\eH(m')^{1/4}{m'}^2}   = \frac{2i^{1/4}}{{m'}^2} \leqslant  
  \frac{2i^{1/4}}{{m}^2} = \rh(m).
 \end{equation*} 
If $m\in [m_i,m_{i+1})$ and $m'\in [m_{i+1},m_{i+2})$,
then 
\begin{equation*}
   \rh(m')= \frac{2(i+1)^{1/4}}{{m'}^2} <    
 \frac{2i^{1/4}} {m^2} \left(\frac{i+1}{i} \right)^{1/4}
\leqslant 2^{1/4} \rh(m),
\end{equation*}
since $(1+i)/i\leqslant 2$. In the remaining case 
$m\in [m_i,m_{i+1})$ and $m'\in [m_{i'},m_{i'+1})$, where 
$i'=i+j \geqslant i+2$, and $m$ and $m'$ satisfy 
$$
m'\geqslant m_{i+j}\geqslant 2^{j-1} m_{i+1} >  2^{j-1} m.
$$
It follows that 
\begin{equation*}
   \rh(m')= \frac{2{i'}^{1/4}}{{m'}^2} \leqslant   
 \frac{2(i+j)^{1/4}}{2^{2j-2}\, m^2} =
\frac{2i^{1/4}}{m^2} \,2^{-2j+2} \left(\frac{i+j}{i}
 \right)^{1/4} < \frac{3^{1/4}}{4}\, \rh(m) <\rh(m)
\end{equation*}
 for $i\geqslant 1, j\geqslant 2$.

\item 
To establish dyadic decay, first suppose $2^r, 2^{r+1}\in
  [m_i,m_{i+1})$.   Then 
  \begin{equation*}
      \rh(2^{r+1}) = \frac{2i^{1/4}}{2^{2(r+1)} } =
      \frac24\frac{i^{1/4}}{2^{2r}} =  \frac14\rh(2^r) .
  \end{equation*}
  Next suppose $2^{r}\in [m_i,m_{i+1})$ and $2^{r+1}\notin
  [m_i,m_{i+1})$. Then since $m_{i+2}\geqslant 2m_{i+1}$, it follows that
  $2^{r+1}\in [m_{i+1},m_{i+2})$ and
\begin{equation*}
  \rh(2^{r+1}) =   \frac{2(i+1)^{1/4}}{2^{2(r+1)}} =\frac24 
 \frac{i^{1/4}}{ 2^{2r}} \left(\frac{i+1}{i}\right)^{1/4}.
\end{equation*}
But $1<(1+i)/i\leqslant 2$ for $i\in\N$, whence for each $r\in\N$,
\begin{equation*}
  \frac14\,\rh(2^r) < \rh(2^{r+1}) 
\leqslant \frac{2^{1/4}}{4}\,\rh(2^{r})< \rh(2^{r}).
\end{equation*}
 Thus $\rh$ decays dyadically (see~\eqref{eq:ubi2decay}).  
  \end{enumerate}
\end{proof}

The main part of the proof is to use the quaternionic Dirichlet
Theorem (Theorem~\ref{thm:QDT}) to establish that the Hurwitz
rationals $\hurat$ form a ubiquitous system.
\begin{lemma}
  \label{lem:fQubi}
  The Hurwitz rationals $\hurat$ in $\overline \Delta$ are ubiquitous
  with respect to the function $\rh$ and the weight given by $\lfloor \q\rfloor
  = |\q|_2$.
\end{lemma}
\begin{proof}
  By the uniform Dirichlet theorem for $\Qu$ (Theorem~\ref{thm:QDT}),
  any point $\xi$ in $B_0$ in $\overline\Delta$ can be approximated
  with an error $2/(|\q|_2N)$ for some $\q$ with $|\q|_2\leqslant N$.
  Thus for each $N\in\N$,
\begin{equation*}
  \label{eq:2}
  B_0\subseteq \kern-2mm\bigcup_{1\leqslant |\q|_2\leqslant N} 
\kern-2mmB\left(\cR_\q;\frac2{|\q|_2N}\right)
\end{equation*}
and so
\begin{equation*}
  B_0=B_0\cap \left(\bigcup_{1\leqslant|\q|_2\leqslant N}  
\kern-2mmB\left(\cR_\q;\frac2{|\q|_2N}\right) \right),
\end{equation*}
where we recall $\cR_\q =\{\pq\in \overline \Delta\}$.

To remove the dependence of the radius on the denominator $\q$, we
select `large' denominators $\q$ with $ \vpi(N)\leqslant |\q|_2\leqslant N, $
where $\vpi\colon \N\to (0,\infty)$ is given by
\begin{equation}
  \label{eq:vpi}
\vpi(m)=\eH(m)^{1/4} m,
\end{equation}
and where, by Lemma~\ref{lem:eta}, $\vpi(m)\to\infty$ as $m\to\infty$.
We remove Hurwitz rationals with `small' denominators as follows.  Let
$E(N)$ be the set of $\xi\in B_0$ with `small' denominator
approximants $\pq$, $1\leqslant |\q|_2<\vpi(N)$ with $|\xi-\pq|<2(|\q|_2
N)^{-1}$. Then
$  B_0=E(N)\cup (B_0\setminus E(N))$
 and  
\begin{equation*}
  E(N)   \subseteq\kern-2mm\bigcup_{1\leqslant |\q|_2 <\vpi(N)} 
 \kern-2mmB\left(\cR_{\q},\frac2{N|\q|_2}\right).
\end{equation*}
By~\eqref{eq:|BRq|} and other estimates
in~\S\ref{subsec:resnearesets}, the Lebesgue measure of $E(N)$
satisfies
\begin{align*}
  |E(N)| &\leqslant \left|\bigcup_{1\leqslant |\q|_2 <\vpi(N)}
    \kern-2mmB\left(\cR_{\q},\frac2{|\q|_2N}\right)\right| \leqslant
  \sum_{1\leqslant|\q|_2< \vpi(N)}
  \kern-2mm\left|B\left(\cR_{\q},\frac2{|\q|_2N}\right))\right| \\
  & \leqslant \kern-2mm\sum_{1\leqslant|\q|_2< \vpi(N)}
  \frac{2^4}{|\q|_2^4N^4}|\q|_2^4 = \frac{2^4}{N^4}
  \sum_{1\leqslant|\q|_2< \vpi(N)} \kern-4mm1
  \\
  &\ll N^{-4}\kern-2mm\sum_{1\leqslant m<\vpi(N)} \kern-2mmm^3 \ll
  N^{-4} \vpi(N)^{4}.
\end{align*}
Since $\vpi(N)=\eH(N)^{1/4}N$, it follows that
$\vpi(N)/N=o(1)$. Thus $|E(N)|\to 0$
and $|B_0\setminus E(N)|\to |B_0| $ as $N\to \infty$.  But by
definition and by~\eqref{eq:vpi}, for each $\xi\in
B_0\setminus E(N)$, there exist $\p,\q\in\hur$ with $\vpi(N)\leqslant
|\q|_2\leqslant N$ such that
\begin{equation*}
|\xi-\pb\q^{-1}|<\frac{2}{|\q|_2N} \leqslant  \frac2{\vpi(N)N} 
= \frac2{\eH(N)^{1/4} N^2} = \rh(N) 
\end{equation*}
by~\eqref{eq:rH} and \eqref{eq:vpi}. 
Moreover by Lemma~\ref{lem:rH}, $\rh$ is dyadically decaying.  Now
\begin{equation*}
  B_0\setminus E(N) 
\subseteq B_0\cap \left(\bigcup_{\vpi(N)\leqslant |\q|_2\leqslant N} 
\kern-2mmB(\cR_{\q},\rh(N)) \right)
\subseteq B_0\cap \left(\bigcup_{1\le |\q|_2\le N} 
\kern-2mmB(\cR_{\q},\rh(N)) \right)
\end{equation*}
and it follows that for $N$ sufficiently large, 
\begin{equation*}
  \label{eq:measest}
  \left|B_0\cap \kern-2mm\bigcup_{1\le |\q|_2\le N} 
\kern-2mmB(\cR_\q,\rh(N))\right|
\geqslant |B_0\setminus E(N)|
\geqslant \frac12\,|B_0|\ (\gg r^4),
\end{equation*}
whence by~\eqref{eq:ubidef} the Hurwitz rationals $\hurat$ are
ubiquitous with respect to the function $\rh$ given by~\eqref{eq:rH}
and the weight $|\cdot|_2$.   \end{proof}

Note that the Hausdorff dimension of $\cV(\Psi)$ in terms of the lower
order of $\Psi$ can be obtained with less difficulty from this
ubiquity result using the methods in~\cite{MDAMshort,DRV90a}.  To
determine the measure requires the extra power of the
Beresnevich-Velani Theorem.  

We now state the specialisation of Theorem~\ref{thm:BV} to $\quat$ and
to Lebesgue and Hausdorff measure.  This theorem unites the divergent
cases of the quaternionic \K{} and \Ja{} theorems.

\begin{theorem}
\label{thm:qbv} 
Let $\Omega=\overline \Delta\subset \Qu$ and $J=\hur\setminus \{0\}$,
so that $\delta=4$, $\cR=\hurat\cap\overline{\Delta}$, $j=\q,$ \,
$R_j=\cRq$ and $\Lambda(\Psi)=\cV(\Psi)$. Let $f$ be a dimension
function with $f(x)/x^4$ increasing and let $\rh$ be given
by~\eqref{eq:rH}, so that $\hurat\cap\overline{\Delta}$ is a ubiquitous
system with respect to the weight $\lfloor \q\rfloor=|\q|_2$ and 
$\rh$. Suppose the ubiquity sum 
   \begin{equation}
      \label{eq:qubiksum}
 \sum_{r=1}^\infty
     \frac{f(\Psi(2^r))}{\rh(2^r)^{4}} 
    \end{equation}
diverges. If  $f(x)=x^4$, then 
 \begin{equation}
\label{eq:H4V}
   \sH^4(\cV(\Psi))=  \sH^4(\overline{\Delta})
= 2^5\pi^{-2} 
 \end{equation}
and if $f(x)/x^4\to\infty$ as $x\to 0$, then
 \begin{equation*}
    \sH^f(\cV(\Psi))= \sH^f(\overline{\Delta}) =\infty.
 \end{equation*}
\end{theorem}

\subsection{The proof of Theorem~\ref{thm:qKT} (Khintchine's theorem
  for $\Qu$)}
 
The proof when the critical
sum~\eqref{eq:critsumLq} converges is given
in~\S\ref{sec:Kconvergent}. 
In the case of divergence,  divergent
dyadic and standard sums need to be compared.
\begin{lemma} 
  \label{lem:comparison}
  Let $\Psi$ be a decreasing approximation function and let $f$ be a
  dimension function. 
If the sum~\eqref{eq:critsumLq} diverges,  
then the ubiquity sum ~\eqref{eq:qubiksum} also diverges. 
\end{lemma}
\begin{proof}
 Take $F(m)= f(\Psi(m))\,m^7$ in Lemma~\ref{lem:eta}.  Then
  by~\eqref{eq:rH}, by the choice of $\eH$ in equation~\eqref{eq:eta}
  and by Lemma~\ref{lem:eta}, the divergence of the sum $
  \sum_{m=1}^\infty f(\Psi(m))\, m^{7} $ implies that the sum
 \begin{equation}
   \label{eq:qubisum}
\sum_{m=1}^\infty f(\Psi(m))\, m^{7}\eH(m) = \sum_{m=1}^\infty
f(\Psi(m))m^7\,  \frac1{m^8\,\rh(m)^4} = \sum_{m=1}^\infty
\frac1{m}\,\frac{f(\Psi(m))}{\rh(m)^4}    
 \end{equation}
 also diverges.  Now since $f(\Psi(m))$ decreases as $m$ increases and
 since $\rh(m')\ll \rh(m)$ when $m'\geqslant m$ (Lemma~\ref{lem:rH}),  
 \begin{eqnarray*}
   \sum_{m=1}^\infty \frac1{m}\,\frac{f(\Psi(m))}{\rh(m)^4} &=& 
 \sum_{r=0}^\infty\,
\sum_{2^r\le m <2^{r+1}}  \frac1{m}\,\frac{f(\Psi(m))}{\rh(m)^4}\\
&\ll& \sum_{r=0}^\infty\,
 2^{-r} f(\Psi(2^r)) \rh(2^{r+1})^{-4}\sum_{2^r\le m <2^{r+1}} 1 \\ 
&\ll&
\sum_{r=0}^\infty f(\Psi(2^r)) \rh(2^{r})^{-4}
 \end{eqnarray*}
and the result follows. 
\end{proof}

Thus the divergence of the critical sum $ \sum_{m=1}^\infty f(\Psi(m))
m^7$~\eqref{eq:critsumLq} implies that
the ubiquity sum~\eqref{eq:qubiksum} also diverges. When the dimension
function $f$ is given by $f(x)=x^4$, it follows from~\eqref{eq:H4V}
and~\eqref{eq:comparable} that $ |\cV(\Psi)|= |\overline{\Delta}| =
1/2$.
 
\subsection{Proofs of \Ja's Hausdorff measure theorem and the
  \JB{} Theorem for $\quat$}

Theorem~\ref{thm:qbv} and Lemma~\ref{lem:comparison} can also be
applied when $f(x)/x^4\to \infty$ as $x\to 0$.  Alternatively the Mass
Transference Principle could be invoked
(see~\S\ref{subsec:appnpowerlaw}).

\paragraph{\Ja's Hausdorff measure theorem (Theorem~\ref{thm:qJT}).}
 Recall from~\eqref{eq:Hcritsum} the definition of the critical 
      sum: 
\begin{equation*}
   \sum_{m=1}^\infty m^7f(\Psi(m)).
\end{equation*}

\paragraph{The case when the critical sum converges:}
 By~\eqref{eq:cover}, for each $N=1,2,\dots$, the family of  balls
\begin{equation*}
  \{ B(\pq,\Psi(\vert \q\vert_2))  
\colon |\p|_2\le |\q|_2, \vert \q\vert_2 \geqslant N\}
\end{equation*}
is a cover for $\cV(\Psi)$.  Hence by~\eqref{eq:Hfmeasure}, for each
$N=1,2,\dots$, the Hausdorff $f$ measure of $\cV(\Psi)$ satisfies
\begin{eqnarray*}
  \sH^f( \cV(\Psi) ) & \leqslant& 
  \sum_{m=N}^\infty\kern1mm
\sum_{ m\leqslant |\q|_2 < m+1}\kern1mm
\sum_{|\p|_2 \le
     |\q|_2} 
  f(\di B(\pq, \Psi(\vert \q\vert_2))) \\ &\ll&
   \sum_{m=N}^\infty\kern1mm
\sum_{m\leqslant|\q|_2< m+1}
  \kern-2mm|\q|_2^4 \ f(2\Psi(|\q|_2)) \ll
  \sum_{m=N}^\infty m^4 \, f(2\Psi(m)) 
 \kern-3mm\sum_{m\leqslant|\q|_2< m+1} \kern-3mm1 \\ &\ll&
 \sum_{m=N}^\infty m^7 \, f(2\Psi(m)).
  \end{eqnarray*}
 But by hypothesis, $f(x)/x^4$ decreases as $x$ increases and so 
\begin{eqnarray*}
  \sH^f( \cV(\Psi) ) & \ll &
  \sum_{m=N}^\infty m^7 \, f(2\Psi(m))\, (2\Psi(m))^{-4} \, (2\Psi(m))^{4}  \\
  &\ll&  \sum_{m=N}^\infty m^7\, f(\Psi(m)) \, (\Psi(m))^{-4}  \,
  2^4(\Psi(m))^{4}\\
  &\ll&  \sum_{m=N}^\infty m^7 \, f(\Psi(m)).
 \end{eqnarray*}
 Thus $\sH^f(\cV(\Psi))=0$ when 
  $\sum_{m=1}^\infty m^7 \ f(V(\Psi))$ converges.

\paragraph{The case when the critical  sum  diverges:}
Lemma~\ref{lem:comparison} implies that the
  ubiquity sum~\eqref{eq:qubiksum} also diverges.  Hence by
Theorem~\ref{thm:qbv}, 
  \begin{equation*}
    \sH^f(\cV(\Psi))= \sH^f(\overline \Delta) = \infty
  \end{equation*}
  when $f(x)/x^4\to\infty$ as $x\to 0$,
which is Theorem~\ref{thm:qJT}. 

\paragraph{Theorem~\ref{thm:qJsm} and the \JB{} Theorem
  (Theorem~\ref{thm:qJBT}).}
The Hausdorff
  $s$-measure result follows by putting $f(x)=x^{s}$.

The Hausdorff
  dimension is the point of discontinuity of $\sH^s(\cW_v)$; this
  occurs at $s=8/v$.

\subsection{ Simultaneous Diophantine approximation in $\R^4$} 
\label{subsec:sdaR4}

 The theorems of Dirichlet, \K, \Ja{} and \JB{} 
on simultaneous \DA{}  in 4-dimensional euclidean space $\R^4$ are stated  
for comparison with quaternions. 
 First, Dirichlet's theorem in $\R^4$~\cite{HW} is stated. 
\begin{theorem}
  For each $\alpha=(\alpha_1,\alpha_2,\alpha_3,\alpha_4)\in\R^4$ and
  $N\in \N$, there exists a $\pb=(p_1,p_2,p_3,p_4)$ in $\Z^4$,
  $q\in\N$ such that
  \begin{equation*}
    \max_{1\le
      m\le4}\left\{\left|\alpha_m-\frac{p_m}{q}\right|\right\}
= \left|\alpha - \frac{\pb}{q}\right|_\infty 
<\frac1{qN^{1/4}}.
  \end{equation*}
Moreover there are infinitely many $\pb\in\Z^4, q\in\N$ such that
\begin{equation*}
   \left|\alpha - \frac{\pb}{q}\right|_\infty <\frac1{q^{5/4}}.
\end{equation*}
\end{theorem}

In the more general form of approximation, write $W^{(4)}(\Psi)$ for the
set of $\Psi$-approximable points  in $\R^4$, \ie, points $\alpha$ such that
\begin{equation*}
  \label{eq:sKT}
  \left|\alpha-\frac{\pb}{q}\right|_\infty<\Psi(q)
\end{equation*}
for infinitely many $\pb\in\Z^4$ and $q\in\N$.  Khintchine's theorem for the
set $W^{(4)}(\Psi) $ takes the form
\begin{theorem}
  \label{thm:realKTa}
   The Lebesgue measure of
  $W^{(4)}(\Psi)$ is null or full according as the critical sum
  \begin{equation*}
    \label{eq:Ksuma}
    \sum_{m=1}^\infty m^4 \Psi(m)^4
  \end{equation*}
  converges or diverges.
\end{theorem}

Gallagher~\cite{Gallagher65} showed that $\Psi$ need not be decreasing
in dimensions $\geqslant 2$, and Pollington \& Vaughan established
that the Duffin-Schaeffer Conjecture also holds in this
case~\cite{PV90}.  \Ja's Hausdorff $f$-measure
result~\cite[Theorem~DV, pg.~66]{BDV06} is now stated for
$W^{(4)}(\Psi)$.

\begin{theorem}[\Ja]
  Let $f$ be a dimension function such that
  $f(x)/x^{4}$ decreases as $r$ increases and $f(x)/x^{4}\to\infty$ as
  $x\to 0$.  Then
  \begin{equation*}
    \sH^f(W^{(4)}(\Psi)) =
    \begin{cases} 0   & \text{ when } \sum_{r=1}^\infty r^4f(\Psi(r)) < \infty \\
    \infty    & \text{ when } \sum_{r=1}^\infty r^4f(\Psi(r)) 
                         = \infty \ \text{ and } \Psi \text{ decreasing}.
  \end{cases}
  \end{equation*}
\end{theorem}
As in the case for $\R$, the two results can be combined into a single
`\K-\Ja' theorem. 

Let $W_v$ denote the set of $\Psi$-approximable points in $\R^4$ when
$\Psi(x)=x^{-v}$.  The \JB{} theorem on simultaneous \DA{} in $\R^4$
follows by taking the dimension function $f(x)=x^{-s}, s>0$.

\begin{corollary}
  \begin{equation*}
  \hdim (W_v)=  
 \begin{cases}     \frac{5}{v}     &  \text{ when } v\le 5/4 \\
                      4 &  \text{ when } v\ge 5/4 .
  \end{cases}
 \end{equation*}
\end{corollary}

 It is evident that exponents in the sums and the \HD {} are quite
  different.  Note that the the identitity $\pq = \p\overline\q/n$,
  where $n=q_1^2+\dots+q_4^2$, gives a natural embedding of
  $\cW(\Psi)$ into 
\begin{pspicture}(0,0)(2.1,0.3)
\rput(1.05,0.15){$W^{(4)}(\Psi\circ \vrule width 6mm height 0 mm depth
  0mm)$}
\rput(1.65,0.15){$\sqrt{\vrule width 2mm height 0 mm depth 0mm}$}
\end{pspicture}
(recall
  $\Psi(x)=\Psi([x]))$.  In
  particular $\cW_v\hookrightarrow W^{(4)}_{v/2}$.

  \section{\Ja's theorem for badly approximable quaternions}
\label{sec:BAquats}
 
The set $\fBQ$ of badly approximable quaternions is defined
analogously to the real case in~\S\ref{subsec:QBA} and are quaternions
for which the exponent in Theorem~\ref{thm:QDT} cannot be increased.
As with ubiquitous systems in~\S\ref{sec:ubiquity}, this notion can be
placed in a general setting of a metric space $(X,d)$ with a compact
subspace $\Omega$ which contains the support of a non-atomic finite
measure $\mu$ and a family $\cR=\{R_j\colon j\in J\}$ of resonant
sets, where $J$ is a countable discrete index set (see~\cite{ktv06}).
The Hausdorff dimension of the set $\mathfrak{B}_\Omega$ of badly
approximable points in $\Omega$ can be determined if the following two
conditions on $\mu$ and $\Psi$ hold.

\vspace{0.05in}

 First, for each ball $B(\xi,r)$,  the measure $\mu$  satisfies 
\begin{equation*}
\label{eq:meascondn} 
   a r^\delta\le \mu(B(\xi,r))\le b r^{\delta},
\end{equation*}
where $0<a\le 1\le b$.  This condition is satisfied by Lebesgue
measure and implies that the Hausdorff dimension of $\Omega$ is given
by $\hdim (\Omega)=\delta$.

\vspace{0.05in}
Secondly, for $\kappa >1$ sufficently large, $\Psi$
  satisfies the `$\kappa$-adic' decay condition
\begin{equation*}
\label{eq:Psicondn} 
  \ell(\kappa) \le \frac{\Psi(\kappa^n)}{\Psi(\kappa^{n+1})} \le
  u(\kappa), \ n \in \N,
\end{equation*}
where $\ell(\kappa) \le u(\kappa) $ and $\ell(\kappa) \to \infty$ as
$\kappa\to\infty$ ({\em{cf}}~\eqref{eq:mBrd} in Theorem~\ref{thm:BV}).
 It is convenient to write for each $n\in\N$
\begin{equation*}
  \label{eq:critrat}
\nu_n=\nu_n(\Psi,\kappa):=  
\left(\frac{\Psi(\kappa^n)}{\Psi(\kappa^{n+1})}\right)^\delta .
\end{equation*}

Recall from~\S\ref{subsec:QBA} that a point $\beta \in\Omega$ which
for some constant $c(\beta)>0$ satisfies
\begin{equation*}
  \label{eq:genbad}
  d(\xi,R_j)\geqslant c(\beta) \Psi(\lfloor j\rfloor) {\text{ for all }} j\in J
\end{equation*}
is called $\Psi$-{\em{badly approximable}}.  The
set of $\Psi$-badly approximable points in $X$ will be denoted by
$\mathfrak{B}_X(\Psi)$.
For each $n\in\N$, let 
$\xi\in\Omega$ and write for convenience
\begin{equation*}
B^{(n)}:=B(\xi,\Psi(\kappa^n))=\{\xi'\in\Omega\colon 
d(\xi,\xi')\le \Psi(\kappa^n)\}  
\end{equation*}
and its scaling by $\theta\in (0,\infty)$ as
\begin{equation*}
  \theta B^{(n)}:=B(\xi,\theta\Psi(\kappa^n))=\{\xi'\in\Omega\colon 
d(\xi,\xi')\le \theta \Psi(\kappa^n)\}.
\end{equation*}

Apart from some changes in notation, the following is Theorem~1
in~\cite{ktv06}  and gives conditions under which  the
Hausdorff dimension of the set of $\Psi$-badly approximable points in
$\Omega$ can be obtained.

\begin{theorem}
  \label{thm:gBA}
  Let $(X,d)$ be a metric space and $(\Omega,d,\mu)$ a compact
  subspace of $X$ with a measure $\mu$.  Let the measure $\mu$ and the
  function $\Psi$ satisfy conditions (A) and (B) respectively. For
  $\kappa \geqslant \kappa_0 >1$, suppose there exists some $\theta\in
  (0,\infty)$ so that for $n\in\N$ and any ball $B^{(n)}$, there exists a
  collection $\cC^{(n+1)}$ of disjoint balls $2\theta
  B^{(n+1)}=B(\cb,2\theta \Psi(\kappa^{n+1})) $ 
  in $\theta \, B^{(n)} $, satisfying
\begin{equation}
 \label{eq:lbballs}
 \# \, \cC^{(n+1)} \geqslant K_1 \nu_n
\end{equation}
and 
\begin{equation}
  \label{eq:ubballs}
  \# 
\left\{
2\theta B^{(n+1)}\subset \theta \, B^{(n)}\colon 
\kern-5mm\min_{{\substack{j\in J,\\ \kappa^{n-1}\leqslant \lfloor j\rfloor < \kappa^n}}} 
\kern-5mmd(\cb, R_j)\leqslant 2\theta \Psi(\kappa^{n+1})\right\}
\leqslant K_2\nu_n,
\end{equation}
where $K_1,K_2$ are absolute constants, independent of $\kappa$ and
$n$, with $K_1>K_2>0$.  Furthermore suppose that $\hdim(\cup_{j\in J}
R_j)<\delta$.  Then
\begin{equation*}
  \hdim \mathfrak{B}_\Omega(\Psi) = \delta.
\end{equation*}
\end{theorem}

The general metric space setting is again specialised to $\quat$ to
give the analogue of \Ja's theorem for the Hausdorff measure and
dimension of the set $\QBA$ of badly approximable quaternions.  When
$X=\quat$ and $\Omega=\overline \Delta$, the measure $\mu$ is
4-dimensional Lebesgue measure, $\delta = 4$, the resonant set $R_j$
is the point $\pq\in\hurat$ and $\lfloor j\rfloor=|\q|_2$.  In view of
the exponent $2$ in~\eqref{eq:QDio} being extremal, we can take
\begin{equation*}
  \Psi(|\q|_2)= |\q|_2^{-2}, 
\end{equation*}
so that $\mathfrak{B}_X(\Psi) = \fBQ $.  Thus in this case 
\begin{equation*}
 \nu_n=  \left(\frac{\Psi(\kappa^n)}{\Psi(\kappa^{n+1})}\right)^\delta =
\left(\frac{\kappa^{-2n}}{\kappa^{-2(n+1)}}\right)^4= \kappa^8,
\end{equation*}
whence $ \nu_n$ is independent of $n$ and satisfied~\eqref{eq:Psicondn}. 

\vspace{0.05in} 
Let
\begin{equation*}
  \label{eq:theta}
\theta=2^{-1}\kappa^{-2}  
\end{equation*} 
and let the 4-ball $B^{(n)} = B(\xi,\kappa^{-2n}) $ lie in
$\overline\Delta$. Then the shrunken ball $\theta B^{(n)}=B(\xi,\theta
\kappa^{-2n}) $ has radius $2^{-1}\kappa^{-2(n+1)} $. 
A collection $\cC^{(n+1)}$ of closed disjoint balls in $\theta
B^{(n)}$ is constructed.  Divide the ball $\theta B^{(n)}$ into
hypercubes $H^{(n+1)}$ of side length $\ell=2^{5/4}
\kappa^{-2(n+2)}$. The number of such hypercubes is at least
\begin{equation*}
\frac12 \frac{|\theta B^{(n)}|}{\ell^4} = 
\frac{\pi^2}{4}\, 2^{-4} \kappa^{-8(n+1)}\times 2^{-5}\,\kappa^{8(n+2)}
=  \frac{\pi^2}{2^{11}} \, \kappa^8.
\end{equation*}
Let $\cC^{(n+1)}$ be the collection of balls $2\theta B^{(n+1)}$ of
radius $ \kappa^{-2(n+2)}$, centred at the centre $\cb$ of a hypercube
$H^{(n+1)}$. The number $ \# \cC^{(n+1)}$ of such  balls
satisfies
\begin{equation*}
 \#\, \cC^{(n+1)} \geqslant  2^{-11} \pi^2 \, \kappa^8
\end{equation*}
 and we can
choose $K_1=  \pi^2/2^{11}$ in~\eqref{eq:lbballs}.

\vspace{0.05in} 

The distance between two points in the ball
$\theta B^{(n)}=B(\xi,\theta \kappa^{-2n}) $ is at most
$\kappa^{-2(n+1)}$. Consider two
distinct Hurwitz rationals $\pq,\rs $, where $\kappa^n\le
|\q|_2, |\sv|_2 < \kappa^{n+1}$ and $\kappa>1$.  
By Lemma~\ref{lem:distdiff},
\begin{equation*}
  |\pq - \p'{\q'}^{-1}|\geqslant  |\q|_2^{-1}|\sv|_2^{-1}  >
  \kappa^{-2(n+1)}, 
\end{equation*}
so that $\theta B^{(n)}$ contains at most one Hurwitz rational $\pq$
with $\kappa^n\leqslant |\q|_2 < \kappa^{n+1}$. Thus such a point $\pq$ can
be in at most one ball $2\theta B^{(n+1)}\in \cC^{(n+1)}$.
Hence for quaternions, the
inequality~\eqref{eq:ubballs} reduces to
 \begin{equation*}
 \# \left\{
2\theta B^{(n+1)}\subset \theta \, B^{(n)}\colon 
\pq\in 2\theta B^{(n+1)},\, \kappa^{n}\leqslant |\q|_2 <
      \kappa^{n+1}
\right\}\leqslant 1   < \frac{\pi^2}{2^{12}}\, \kappa^8
 \end{equation*} 
 for $\kappa\geqslant 3$, so that we can choose $K_2= \pi^2/2^{12} <K_1$.
 Finally, since the resonant sets $\pq$ are points,
 $\hdim(\{\pq\}) =0$. It now follows from Theorem~\ref{thm:gBA} that
 $\hdim(\fBQ)=4$, \ie, $\fBQ\subset \quat$ has full Hausdorff dimension.

 As in the classical case, Theorem~\ref{thm:qKT} can be used to show 
that $\fBQ$ is null. 
 Indeed since $\Psi(m)= m^{-2}$, the sum $\sum_{m\in\N}
 \Psi(m)^4m^{7}= \sum_{m\in\N} \, m^{-1}$ diverges. Hence the set of
 $\beta\in \overline\Delta$ satisfying the inequality
   \begin{equation*}
     |\beta-\pq|_2\geqslant \frac{1}{|\q|_2^2} 
   \end{equation*}
   for all but finitely many $\pq$, say $\p^{(m)}(\q^{(m)})^{-1}$,
   $m=1,2,\dots,N=N(\xi)$, is null.  Let
\begin{equation*}
  c(\beta):= \min\{1, |\beta-\p^{(m)}(\q^{(m)})^{-1}|_2
 |\q^{(m)} |_2^2 \colon m=1,2,\dots N\}.
\end{equation*}
Then the set of $\beta\in \overline\Delta$ satisfying the inequality
\begin{equation*}
  |\beta-\pq|_2\geqslant \frac{c(\beta)}{|\q|_2^2}  
\end{equation*}
for all $\pq$ is null. This completes the proof of
Theorem~\ref{thm:QJTBA}, the analogue of \Ja's theorem for
the set $\fBQ$ of badly approximable quaternions.

\section{Acknowledgements}
We are grateful to Victor Beresnevich and Sanju Velani for their
interest, advice and patience, and to the referee for helpful
comments.

\bibliographystyle{plain} 



\newcommand{\noopsort}[1]{} \newcommand{\singleletter}[1]{#1}
\providecommand{\bysame}{\leavevmode\hbox to3em{\hrulefill}\thinspace}
\providecommand{\MR}{\relax\ifhmode\unskip\space\fi MR }
\providecommand{\MRhref}[2]{%
  \href{http://www.ams.org/mathscinet-getitem?mr=#1}{#2}
}
\providecommand{\href}[2]{#2}

\end{document}